\documentclass[a4paper,12pt]{article}
\usepackage{amsmath, amsthm,amssymb, latexsym}
\usepackage{tikz,hyperref,pgf,pgflibraryarrows,enumerate, amscd, amssymb}
\usepackage[all]{xy}
\usepackage[utf8]{inputenc}
\usepackage[T1]{fontenc}
\usepackage{url}
\usepackage{mathbbol}
\usepackage[title]{appendix}

\newtheorem{thm}{Theorem}[section]
\newtheorem{cor}[thm]{Corollary}
\newtheorem{lem}[thm]{Lemma}
\newtheorem{prop}[thm]{Proposition}

\theoremstyle{definition}

\theoremstyle{remark}
\newtheorem{rmk}[thm]{Remark}

\newtheorem{example}[thm]{Example}

\numberwithin{equation}{section}

\newcommand{\Z}{\mathbb{Z}}

\newcommand{\RR}{\mathbb{R}}
\newcommand{\R}{\mathbb{R}}
\newcommand{\TT}{\mathbb{T}}
\newcommand{\T}{\mathbb{T}}

\newcommand{\DD}{D^{c}_{\mu\nu}}
\newcommand{\Ind}{\operatorname{Ind}}
\newcommand{\Id}{\operatorname{Id}}

\newcommand{\EE}{E^c_{\mu\nu}}
\newcommand{\HH}{\mathbb{H}}

\newcommand{\YM}{\operatorname{YM}}

\newcommand{\id}{\operatorname{id}}

\providecommand{\keywords}[1]{{\textit{Keywords and phrases:}} #1}
\providecommand{\classification}[1]{{\textit{2010 Mathematics Subject Classification:}} #1}

\usepackage[margin=0.9in]{geometry}
\AtEndDocument{\bigskip{\small
  \textsc{Sooran Kang : College of General Education, Chung-Ang University, 84 Heukseok-ro, Dongjak-gu, Seoul, Republic of Korea} \par
  \textit{E-mail address}: \texttt{sooran09@cau.ac.kr} \par
  \addvspace{\medskipamount}
  \textsc{Franz Luef : Department of Mathematics, NTNU Norwegian University of Science and
Technology, Trondheim, Norway} \par
  \textit{E-mail address}: \texttt{franz.luef@ntnu.no} \par
\addvspace{\medskipamount}
  \textsc{Judith A. Packer : Department of Mathematics, University of Colorado at Boulder, Boulder, Colorado, 80209-0395, USA.} \par
  \textit{E-mail address}: \texttt{packer@colorado.edu}
}}

\title{Yang-Mills connections on quantum Heisenberg manifolds}

\author{Sooran Kang, Franz Luef and Judith A. Packer}
\date{\today}

\begin{document}

\maketitle

\begin{abstract}

We investigate critical points and minimizers of the Yang-Mills functional $\YM$ on quantum Heisenberg manifolds $\DD$, where the Yang-Mills functional is defined on the set of all compatible linear connections on finitely generated projective modules over $\DD$. 
 A compatible linear connection which is both a critical point and minimizer of $\YM$ is called a Yang-Mills connection. In this paper, we investigate Yang-Mills connections with constant curvature. We are interested in Yang-Mills connections on  the following classes of modules over $\DD$: (i) Abadie's module $\Xi$ of trace $2\mu$ and its submodules; (ii) modules $\Xi^\prime$ of trace $2\nu$; (iii) tensor product modules of the form $P\EE\otimes \Xi$, where $\EE$ is Morita equivalent to $\DD$ and $P$ is a projection in $\EE$. We present a characterization of critical points and minimizers of $\YM,$ and provide a class of new Yang-Mills connections with constant curvature on $\Xi$ over $\DD$ via concrete examples. In particular, we show that every Yang-Mills connection $\nabla$ on $\Xi$ over $\DD$ with constant curvature should have a certain form of the curvature such as $\Theta_\nabla(X,Y)=\Theta_\nabla(X,Z)=0$ and $\Theta_\nabla(Y,Z)=\frac{\pi i}{\mu}\Id_E$.  Also we show that these Yang-Mills connections with constant curvature do not provide global minima but only local minima. We do this by constructing a set of compatible connections that are not critical points but their values are smaller than those of Yang-Mills connections with constant curvature.
 Our other results include: (i) an example of a compatible linear connection with constant curvature on $\DD$ such that the corresponding connection on an isomorphic projective module does not have constant curvature, and (ii) the construction of a compatible linear connection with constant curvature which neither attains its minimum nor is a critical point of $\YM$ on $\DD$. Consequently the critical points and minimizers of $\YM$ depend crucially on the geometric structure of $\DD$ and of the projective modules over $\DD$. 
Furthermore, we construct the Grassmannian connection on the projective modules $\Xi^\prime$ with trace $2\nu$ over $\DD$  and compute its  corresponding curvature. Finally, we construct tensor product connections on $P\EE\otimes \Xi$ whose coupling constant is $2\nu$ and characterize the critical points of $\YM$ for this projective module.

\end{abstract}

\classification{46L05, 46L87, 58B34}


\tableofcontents

\section{Introduction}

Strict deformation quantization of manifolds is a convenient way to construct noncommutative $C^*$-algebras; noncommutative tori are perhaps the best--known example of these. Here, we focus on quantum Heisenberg manifolds (henceforth abbreviated by QHMs) \cite{Rie3}, which are strict deformation quantizations of Heisenberg manifolds $M_c$ where $c$ is a positive integer. The QHMs are the fibers of  a continuous field of  $C^*$-algebras $\{D^{c,\hslash}_{\mu,\nu}\}$, where $\hslash$ denotes the Planck constant representing the deformation parameter and $\mu,\nu\in \R$ are parameters coming from the Poisson structure on the Heisenberg manifolds $M_c$. 

The QHMs give another important family of noncommutative manifolds along with noncommutative tori, but the QHMs are much less well understood. The QHMs were originally introduced as generalized fixed point algebras of certain crossed product $C^*$-algebras \cite{Rie3}. They also can be viewed as crossed products by Hilbert $C^*$-bimodules (also called generalized crossed products) \cite{AEE}, which themselves are examples of Cuntz-Pimsner algebras \cite{Kat1}. Moreover, the QHMs can be also realized as twisted groupoid $C^*$-algebras \cite{KKP}. In a series of papers  \cite{Ab1, Ab2, Ab3}, Abadie has studied the $K$-theory, range of the trace on the set of projections in $\{D^{c,\hslash}_{\mu,\nu}\}$, and Morita equivalence classes of the QHMs. Finally, Gabriel studied cyclic cohomology and index pairings of the QHMs in \cite{G}. 

Connes and Rieffel introduced Yang-Mills theory on noncommutative manifolds and studied it extensively for noncommutative 2-tori \cite{CR}. The main objective of our investigation is furthering the understanding of Yang-Mills theory on QHMs using the framework of \cite{CR} that was originally initiated in \cite{Kang1} by the first author. Later Lee found a minimizing set of the Yang-Mills functionals in \cite{Lee} using the finitely generated projective module studied by Kang in \cite{Kang1} that was  originally constructed by Abadie in \cite{Ab1}. Chakraborty et al discussed the geometry of the QHMs in \cite{CS, Ch1}, and their contributions  in \cite{Ch2} include that the Yang-Mills functional coming from a spectral triple coincides with the Yang-Mills functional coming from $C^*$-dynamical systems as in \cite{CR} for the QHMs.

It is well-known that vector bundles that are isomorphic to one another can be equipped with distinctly different geometric structures. However, the existing literature has not yet shown this phenomenon clearly for the noncommutative analog of vector bundles over the QHMs, whose geometric structure is very different from that of noncommutative 2-tori. In this paper, we provide two examples of compatible linear connections on the QHMs, the first of which shows that a connection can be isomophic to a compatible linear connection on $\Xi$ with a constant curvature, yet not have constant curvature itself. (See Proposition~\ref{prop:conn_ex1} and Proposition~\ref{prop:conn_ex1-not-cp}). The second example shows that a compatible linear connection with constant curvature on a QHM need not give a critical point of the $\YM$ functional or a minimum value for the $\YM$ functional. (See Proposition~\ref{prop:nabla-1-not-cp}). This provides a stark contrast to noncommutative 2-tori, where the minimum value for the $\YM$ functional occurs on connections with constant curvature. (See \cite{CR} and \cite{Rie5-cp}).

Let us briefly summarize the contents of this paper: In Section 2, the QHMs\footnote{Since $D^{c,\hslash}_{\mu,\nu}$ is isomorphic to $D^{c,1}_{\hslash\mu, \hslash\nu},$ from now on, we drop the Planck constant $\hslash$ from our notation.}
$\DD$ and $E^{c}_{\mu\nu}$ are introduced as generalized fixed point algebras and we define Abadie's equivalence bimodule $\Xi$ between these two QHMs. On the QHM $\DD$, we define an action of the Heisenberg group which induces a natural class of derivations on the smooth subalgebra $(\DD)^\infty$. We continue with a discussion of compatible linear connections on the finitely generated projective module $\Xi$ and define the main player of this investigation, the Yang-Mills functional $\YM$ on the set of all compatible linear connections $CC(\Xi)$.

Section~\ref{sec:mot_ex} provides further observations on $\YM$: In Proposition~\ref{prop:conn_ex1} and Proposition~\ref{prop:conn_ex1-not-cp}, we  discuss an example of a  connection with constant curvature on $\Xi$, such that the corresponding connection on a projective module isomorphic to $\Xi$ ceases to have constant curvature.  A natural question is then whether or not the $\YM$  functionals for the quantum Heisenberg manifolds $\DD$ attain minima or have a critical point at every compatible connection on $\Xi$ with constant curvature. This question is quite subtle and difficult, as the following indicates: we give an example of a compatible connection with constant curvature on $\Xi$ at which $\YM$ neither attains a minimum nor attains a critical point in Proposition~\ref{prop:nabla-1-not-cp}, which shows that the answer to our question is a negative one. 

In Section~\ref{sec:YM-conn}, we consider Yang-Mills connections with constant curvature on $\Xi$ over the QHM $\DD$. Our first result  Proposition~\ref{prop:critical} gives a necessary condition for a connection in $CC(\Xi)$ to be a critical point of $\YM$. Then we give necessary and sufficient condition for a connections $\nabla$ to be a critical point of $\YM$ when $\nabla$ has a constant curvature in Theoren~\ref{thm:critical-ioi}. We also discuss characterizations of Yang-Mills connections  with constant curvature on the projective module $\Xi$ in Theorem~\ref{prop:min-cp} and Corollary~\ref{thm:ym-connection}. In particular, we prove that $\nabla$ is a Yang-Mills connection on $\Xi$ with constant curvature if and only if the curvature $\nabla$ is given by $\Theta_\nabla(X,Y)=\Theta_\nabla(X,Z)=0$ and $\Theta_\nabla(Y,Z)=\frac{\pi i}{\mu}\Id_E$, where $\Id_E$ is the identity of $\EE$.
The section gives a new class of Yang-Mills connections on $\Xi$ (see Theorem~\ref{thm:YM-connection} and Example~\ref{ex:YM-conn}), and proves that in the constant curvature case, compatible connections that are minimizers subject to the constraint of constant curvature must also be critical points for the Yang-Mills functional (see Remark~\ref{rmk:const-conn-not-min}), whereas in the nonconstant curvature situation, this need not be the case. In fact, in Section~\ref{sec:non-const}, we construct a set of compatible connections on $\Xi$ with nonconstant curvature in Theorem~\ref{thm:min-but-nc} that are not critical points of $\YM$ but attain smaller values of $\YM$ than that of Yang-Mills connections on $\Xi$ with constant curvature. This shows that the Yang-Mills connections with constant curvature do not give global minima but are minima only subject to the constraint of constant curvature.

Section~\ref{sec:trace_2nu} looks into projective modules over $\DD$ with trace $2\nu$ instead of $\Xi$ which has trace $2\mu$, in particular a submodule $P\Xi$ of $\Xi$ and a balanced tensor product module $P\EE\otimes_{\EE} \Xi$, where $P$ is a projection in $\EE$ with trace $2\nu$. We first construct the Grassmannian connection on $P\Xi$ and compute the corresponding curvature (Proposition~\ref{prop:Grass_conn} and Proposition~\ref{prop:Grass-curvature}). Then we investigate tensor product compatible connections on $P\EE\otimes_{\EE} \Xi$ and show that they are critical points for the Yang Mills functional if and only if the connections on their respective domains are critical points themselves. (Theorem~\ref{thm:tensor-prod}).

 In the Appendix we derive various crucial calculations that show a modified formula $\nabla^0$ for Lee's connection \cite{Lee} in our setting indeed gives a compatible connection on $\Xi$ with constant curvature $\Theta_{\nabla^0}(X,Y)=\Theta_{\nabla^0}(X,Z)=0$ and $\Theta_{\nabla^0}(Y,Z)=\frac{\pi i}{\mu}$.

\subsection*{Acknowledgments}   
S.K.~was supported by Basic Science Research Program through the National Research Foundation of Korea (NRF) funded by the Ministry of Education (\#NRF-2017R1D1A1B03034697). J.P.~was partially supported by an individual  grant from the Simons Foundation (\#316981).

\section{Preliminaries}\label{sec:prelim}

In this section, we briefly review the finitely generated projective module $\Xi$ over the quantum Heisenberg manifold $\DD$ constructed by Abadie \cite{Ab1} and compatible linear connections on $\Xi$ given in \cite{Kang1, Lee}. Note that throughout the paper, when we say ``projective'',  we mean  ``finitely generated projective''.

Quantum Heisenberg manifolds are constructed as follows: Let $M=\RR\times\TT$ and let $\lambda$ and $\sigma$ be the commuting actions of $\mathbb{Z}$ on $M$ defined by
\begin{equation*}
\lambda_p(x,y)=(x+2 p\mu,y+2 p\nu)\quad\text{and}\quad
\sigma_p(x,y)=(x-p,y),
\end{equation*}
where  $\mu, \nu \in \mathbb{R}$, and $p\in \mathbb{Z}$.

Then form the crossed product $C^{\ast}$-algebras $C_b(M)\rtimes_{\lambda}\mathbb{Z}$ and $C_b(M)\rtimes_{\sigma}\mathbb{Z}$ with the usual star-product and involution. Here $C_b(M)$ is the space of continuous bounded functions on $M$, and  $\rho$ and $\gamma$ denote the actions of $\mathbb{Z}$ on $C_{b}(M)\rtimes_{\lambda}\mathbb{Z}$ and $C_{b}(M)\rtimes_{\sigma}\mathbb{Z}$ given by, for $\Phi,\Psi \in C_c(M\times \mathbb{Z})$,
\begin{equation}\label{eq:fixed-cond}
\begin{split}
(\rho_k\Phi)(x,y,p)&=\overline{e}(ckp(y- p\nu))\Phi(x+k,y,p),\\
(\gamma_k\Psi)(x,y,p)&=e(cpk(y- k\nu))\Psi(x-2 k\mu,y-2 k\nu,p),
\end{split}\end{equation}
 where $k, p \in \mathbb{Z}$, and $e(x)=exp(2\pi ix)$ for any real number $x$.
Then these actions $\rho$, $\gamma$ are proper in the sense of \cite{Rie4-proper}.
 The generalized fixed point algebra of $C_{b}(M)\rtimes_{\lambda}\mathbb{Z}$ by the action $\rho$, denoted by $\DD$, is the closure of $\ast$-subalgebra $D_0$ in the multiplier algebra of $C_{b}(M)\rtimes_{\lambda}\mathbb{Z}$  consisting of functions $\Phi \in C_c(M\times \mathbb{Z})$, which have compact support on $\mathbb{Z}$ and satisfy $\rho_k(\Phi)=\Phi$ for all $k \in \mathbb{Z}$. The $C^*$-algebra $\DD$ is the quantum Heisenberg manifold we are studying in this paper. 

We now introduce another quantum Heisenberg manifold. The generalized fixed point algebra of $C_{b}(\mathbb{R}\times \mathbb{T})\rtimes_{\sigma}\mathbb{Z}$ by the action $\gamma$, denoted by $E^{c}_{\mu\nu}$, is the closure of $\ast$-subalgebra $E_0$ in the multiplier algebra of $C_{b}(\mathbb{R}\times \mathbb{T})\rtimes_{\sigma}\mathbb{Z}$  consisting of functions $\Psi \in C_c(\mathbb{R}\times \mathbb{T}\times \mathbb{Z})$, with compact support on $\mathbb{Z}$ and satisfying $\gamma_k(\Psi)=\Psi$ for all $k \in \mathbb{Z}$. Note that we can consider $\EE$ to be a quantum Heisenberg manifold, since $\EE$ has been shown to be isomorphic to $D^c_{\frac{1}{4\mu}, \frac{\nu}{2\mu}}$ in Proposition~2 of \cite{Kang1}.

According to the main theorem in \cite{Ab1}, these generalized fixed point algebras $D^{c}_{\mu\nu}$ and $E^{c}_{\mu\nu}$ are strongly Morita equivalent and thus there exists an equivalence bimodule $\Xi$ between $D^{c}_{\mu\nu}$ and $E^{c}_{\mu\nu}$. Concretely, $\Xi$ is the completion of $C_c(\mathbb{R}\times \mathbb{T})$ with respect to either one of the norms induced by one of the $D^{c}_{\mu\nu}$ and $E^{c}_{\mu\nu}$-valued inner products, $\langle\cdot,\cdot\rangle_R^D$ and $\langle\cdot,\cdot\rangle_L^E$ respectively, given by
According to the main theorem in \cite{Ab1}, these generalized fixed point algebras $D^{c}_{\mu\nu}$ and $E^{c}_{\mu\nu}$ are strongly Morita equivalent and thus there exists an equivalence bimodule $\Xi$ between $D^{c}_{\mu\nu}$ and $E^{c}_{\mu\nu}$. Concretely, $\Xi$ is the completion of $C_c(\mathbb{R}\times \mathbb{T})$ with respect to either one of the norms induced by one of the $D^{c}_{\mu\nu}$ and $E^{c}_{\mu\nu}$-valued inner products, $\langle\cdot,\cdot\rangle_R^D$ and $\langle\cdot,\cdot\rangle_L^E$ respectively, given by
\begin{equation*}\label{D-value-inner}
\langle f,g \rangle_R^D(x,y,p)=\sum_{k \in \mathbb{Z}}\overline{e}(ckp(y-p\nu))f(x+k,y)\overline{g}(x-2 p\mu+k,y-2 p\nu)\,,
\end{equation*}
\begin{equation*}\label{E-value-inner}
\langle f,g \rangle_L^E(x,y,p)=\sum_{k \in \mathbb{Z}}e(cpk(y- k\nu)\overline{f}(x-2 k\mu,y-2 k\nu)g(x-2 k\mu+p,
y-2 k\nu),
\end{equation*}
 where $f,g \in C_c(\mathbb{R}\times \mathbb{T})$ and $k,\,p \in \mathbb{Z}$.
 Note that the $\DD$-valued inner product $\langle \cdot, \cdot \rangle_R^D$ is conjugate linear in the second variable, i.e. $\langle f, \alpha g\rangle_R^D=\overline{\alpha}\, \langle f, g \rangle_R^D$ for $f,g\in \Xi$, $\alpha\in \mathbb{C}$.
The left and right action of $E^{c}_{\mu\nu}$ and $D^{c}_{\mu\nu}$ on $\Xi$ are given by
\begin{equation*}\label{left-action}
(\Psi\cdot f)(x,y)=\sum_{q \in\mathbb{Z}}\overline{\Psi}(x,y,q)f(x+q,y),
\end{equation*}
\begin{equation*}\label{right-action}
(g\cdot\Phi)(x,y)=\sum_{q\in\mathbb{Z}}g(x+2 q\mu,y+2 q\nu)\overline{\Phi}(x+2 q\mu,y+2 q\nu,q),
\end{equation*}
for $\Psi \in E_0$ , $\Phi \in D_0$ and $f, g \in \Xi$. 

Let $H$ be the reparametrized Heisenberg group given by Rieffel in \cite{Rie3} as follows: for $x,y,z\in \RR$ and a positive integer $c$, let
\[
(x,y,z):=\begin{pmatrix} 1 & y & z/c \\ 0 & 1 & x \\ 0 & 0 & 1 \end{pmatrix}.
\]
Then we can identify $H$ with $\RR^3$ with the product
\[
(x,y,z)(x'y'z')=(x+x', y+y', z+z'+cyx').
\]
The action $L$ of $H$ on the quantum Heisenberg manifold $\DD$ is given by
 \[
(L_{(r,s,t)}\Phi)(x,y,p)=e(p(t+cs(x-r-p\mu)))\Phi(x-r,y-s,p).
 \]
The smooth subalgebra $(\DD)^\infty$ of $\DD$ is given by
 \[
 (\DD)^\infty=\{d\in \DD: h\mapsto L_h(d)\;\;\text{is smooth in norm for $h\in H$}\},
 \]
The infinitesimal form of $L$ gives an action $\delta$ of $\mathfrak{h}$ on $(\DD)^\infty$, where $\mathfrak{h}$ is the corresponding Heisenberg Lie algebra of the reparametrized Heisenberg group $H$. In particular, we let $X, Y, Z$ be the basis of $\mathfrak{h}$ given by
 \begin{equation}\label{eq:basis-XYZ}
 X=(0,1,0)=\begin{pmatrix} 0 & 1& 0 \\ 0 & 0 & 0 \\ 0 & 0 & 0 \end{pmatrix}, \;\; Y=(1,0,0)=\begin{pmatrix} 0 & 0 & 0 \\ 0 & 0 & 1 \\ 0 & 0 & 0 \end{pmatrix}, \;\; Z=(0,0,1)=\begin{pmatrix} 0 & 0 & 1/c \\ 0 & 0 & 0 \\ 0 & 0 & 0 \end{pmatrix}
 \end{equation}
 and then we have $[X,Y]=cZ$. The corresponding derivation $\delta$ on $(\DD)^\infty$ are given by
 \[
 \delta_X(\Phi)(x,y,p)=2\pi icp(x-p\mu)\Phi(x,y,p)-\frac{\partial \Phi}{\partial y}(x,y,p),
 \]
\[
\delta_Y(\Phi)(x,y,p)=-\frac{\partial \Phi}{\partial x}(x,y,p),
\]
and
\[
\delta_Z(\Phi)(x,y,p)=2\pi ip\,\Phi(x,y,p).
\]
According to Lemma~1 of \cite{C1}, there is a dense left $(\EE)^\infty$--right $(\DD)^\infty$ submodule $\Xi^\infty$ of the left $\EE$--right $\DD$ equivalence bimodule $\Xi$, which is a projective and finitely generated left $\EE$--module and a finitely generated, projective right $\DD$-module. For notational simplicity we omit the superscript  ``$\infty$" from smooth spaces of $C^*$-algebras and projective modules over them from now on. 

For $\Xi$ and $\mathfrak{h}$ above, we say that a linear map $\nabla:\Xi \to \Xi \otimes \mathfrak{h}$ is a linear connection if it satisfies
\begin{equation}\label{eq:nabla-der}
\nabla_X(\xi \cdot \Phi)=(\nabla_X(\xi))\cdot \Phi + \xi \cdot (\delta_X(\Phi)),
\end{equation}
for all $X\in \mathfrak{h}$, $\xi\in \Xi$ and $\Phi\in \DD$.
We say that a linear connection is compatible with respect to the inner product (often called the  Hermitian metric) $\langle \cdot, \cdot \rangle_R^D$ if 
\begin{equation}\label{eq:nabla-comp}
\delta_X(\langle \xi,\eta\rangle_R^D)=\langle \nabla_X \xi, \eta\rangle_R^D + \langle \xi, \nabla_X \eta\rangle_R^D.
\end{equation}
The curvature of a compatible connection $\nabla$ is defined to be the alternating bilinear form $\Theta_\nabla$ on $\mathfrak{h}$, given by
\[
\Theta_\nabla(X,Y)=\nabla_X \nabla_Y - \nabla_Y \nabla_X -\nabla_{[X,Y]}
\]
for $X,Y\in \mathfrak{h}$. From now on, we say ``connection'' when we mean ``linear connection''. 
We denote the set of compatible connections on $\Xi$ by $CC(\Xi)$. 

To define the Yang-Mills functional $\YM$ on $CC(\Xi)$, we need to introduce some more structure. Let $\tau$ be a faithful $L$-invariant trace on $\DD$, where $L$ is the action of Heisenberg group on $\DD$. Also we define the trace $\tau_E$ on $\EE$ induced by $\tau$ by 
\[
\tau_E(\langle \xi,\eta\rangle_L^E)=\tau(\langle \eta,\xi\rangle_R^D).
\]
According to \cite{Rie1}, there is a faithful $L$-invariant trace on $\DD$ given by
\[
\tau(\Phi)=\int_{\TT^2} \Phi(x,y,0)\, dx\, dy
\]
for $\Phi\in \DD$, and one can show that
\begin{equation}\label{eq:trace_E}
 \tau_E(\Psi)=\int_0^{2\mu}\int_0^1 \Psi(x,y,0)\,  dy\, dx 
\end{equation}
for $\Psi \in (\EE)_0$ and $\mu>0$. 

The Yang-Mills functional $\YM$ is defined on $CC(\Xi)$ by
\begin{equation}\label{eq:YM-functional}
\YM(\nabla)=-\tau_E(\{\Theta_\nabla,\Theta_\nabla\}_E),
\end{equation}
where $\{\cdot,\cdot\}_E$ is a bilinear form given by
\[
\{\Phi, \Psi\}_E=\sum_{i<j}\Phi(Z_i,Z_j)\Psi(Z_i,Z_j),
\]
for alternating $\EE$-valued 2-forms $\Phi, \Psi$, where $\{Z_1,Z_2,Z_3\}$ is a basis for $\mathfrak{h}$.

We say that a compatible connection $\nabla$ attains a global minimum for $\YM$ if $\YM(\nabla') \ge \YM(\nabla)$ for any other connection $\nabla'\in CC(\Xi)$, and we say that a compatible connection $\nabla$ with constant curvature attains a local minimum for $\YM$ if $\YM(\nabla')\ge \YM(\nabla)$ for any other connection $\nabla'\in CC(\Xi)$ with constant curvature.

Let $\mathcal{U}(\EE)$ be the group of unitary elements of $\EE$, acting on $CC(\Xi)$ by conjugation. i.e. for $u\in \mathcal{U}(\EE)$, $\nabla\in CC(\Xi)$, we define $\mathcal{G}_u(\nabla)$ by
\[
(\mathcal{G}_u(\nabla))_X (\xi)=u\cdot (\nabla_X(u^\ast\cdot \xi))
\]
for $\xi\in \Xi$ and $X\in \mathfrak{h}$. Then it is straightforward to check that $\mathcal{G}_X(\nabla)\in CC(\Xi)$. Also we have
\[
\Theta_{\mathcal{G}_u(\nabla)}(X,Y)=u\,\Theta_\nabla(X,Y) u^* \quad \;\; \text{for $X,Y\in \mathfrak{h}$}
\]
and
\[
\{\Theta_{\mathcal{G}_u(\nabla)}, \Theta_{\mathcal{G}_u(\nabla)}\}=u\, \{\Theta_\nabla, \Theta_\nabla\} u^*.
\]
Thus it follows that 
\[
\YM(\mathcal{G}_u(\nabla))=\YM(\nabla)
\]
for $u\in \mathcal{U}(\EE)$ and $\nabla\in CC(\Xi)$, and hence the Yang-Mills functional $\YM$ is well-defined on the quotient space $CC(\Xi)/\mathcal{U}(\EE)$. One of the main concerns of Yang-Mills theory is to describe the set of minima for $\YM$ on this quotient space, which is called the moduli space for $\Xi$.

Two different compatible connections $\nabla^G$ and $\nabla^0$ for $\delta$ on $\Xi$ with respect to $\langle \cdot, \cdot\rangle_R^D$ have been found in \cite{Kang1} and \cite{Lee}. The former $\nabla^G$ is the Grassmannian connection on $\Xi$ given by, for all $X\in \mathfrak{h}$
\begin{equation}\label{eq:Gr-conn}
\nabla^G_X(\xi)=R\cdot \delta_X(\langle R, \xi\rangle_R^D),
\end{equation}
where $R\in \Xi$,  $\langle R, R\rangle_L^E=\Id_E$, the identity of $\EE$, and $\langle R, R\rangle_R^D$ is a projection in $\DD$.

With the specific function $R$ described in \cite{Kang1}, we have
\[\begin{split}
&\Theta_{\nabla^G}(X,Y)(x,y,p)=f_1(x)\delta_0(p)\\
&\Theta_{\nabla^G}(X,Z)(x,y,p)=0\\
&\Theta_{\nabla^G}(Y,Z)(x,y,p)=f_2(x)\delta_0(p)
\end{split}\]
where $f_1$ and $f_2$ are smooth skew-symmetric periodic functions. (See the details in \cite{Kang1}).

The compatible connection $\nabla^0$ on $\Xi$ is given by
\begin{equation}\label{eq:min_conn}
\begin{split}
&(\nabla^0_X\xi)(x,y)=-\frac{\partial \xi}{\partial y}(x,y)+\frac{\pi ci}{2\mu}x^2f(x,y)\\
&(\nabla^0_Y\xi)(x,y)=-\frac{\partial \xi}{\partial x}(x,y)\\
&(\nabla^0_Z\xi)(x,y)=\frac{\pi i x}{\mu}\xi(x,y).
\end{split}
\end{equation}
Since our setting differs slightly from the setting given in \cite{G} and \cite{Lee}, for readers' convenience we verify that $\nabla^0$ above indeed is a compatible linear connection on $\Xi$ with respect to $\langle \cdot, \cdot\rangle_R^D$ in Proposition~\ref{App:prop:cc} in Appendix~\ref{app:sec:cc}.

Note that $\nabla^0$ is shown to be a minimizer of $\YM$ in \cite{Lee}, but in fact it turns out to be a local minimizer of $\YM$, in the sense that $\YM(\nabla^0) \le \YM(\nabla)$ for any other connection $\nabla$ with constant curvature (see Theorem~\ref{prop:min-cp}), and it is a critical point of $\YM$. When we consider the restricted moduli space of connections with constant curvature, Lee's connection $\nabla^0$ will give a global minimum. However, when we extend the moduli space to consider all compatible connections, $\nabla^0$ gives no longer a global minimum but only a local one.
Also note that the connection $\nabla^0$ is the only minimizing YM connection known up to this point.

The curvature of $\nabla^0$ is given by
\begin{equation}\label{eq:min_curv}
\Theta_{\nabla^0}(X,Y)=0,\;\;\Theta_{\nabla^0}(X,Z)=0,\;\;\Theta_{\nabla^0}(Y,Z)=\frac{\pi i}{\mu}\Id_E,
\end{equation}
where $\Id_E(x,y,p)=\delta_0(p)$. Note that the constant curvature here looks a bit different from the one given in \cite{Lee} since our setting is different. In particular, since the curvature of a compatible connection on $\Xi$ over $\DD$ is skew-symmetric and $\EE$-valued,  we thus obtain $\Theta_{\nabla^0}(Y,Z)=\frac{\pi i}{\mu}\Id_E$; see Proposition~\ref{propA:curv} in Appendix~\ref{app:sec:cc} for details.

According to Theorem~1.1 of \cite{Rie5-cp}  and Section 5 of \cite{Kang1}, a compatible connection $\nabla$ on $\Xi$ with curvature $\Theta_\nabla$ is a critical point of $\YM$ exactly when $\nabla$ satisfies the following equations:
\begin{equation}\label{eq:critical_point}
\begin{split}
&(1)\;[\nabla_Y, \Theta_\nabla(X,Y)]+[\nabla_Z,\Theta_\nabla(X,Z)]=0,\\
&(2)\;[\nabla_X,\Theta_\nabla(Y,X)]+[\nabla_Z,\Theta_\nabla(Y,Z)]=0,\\
&(3)\;[\nabla_X,\Theta_\nabla(Z,X)]+[\nabla_Y, \Theta_\nabla(Z,Y)]-c\Theta_\nabla(X,Y)=0.
\end{split}\end{equation}

\section{Two observations}\label{sec:mot_ex}

\subsection{Isomorphic projective modules with different geometric invariants}
 We now show that the geometry on isomorphic projective modules can be quite different. Namely, we provide an example of a connection with  constant curvature on $\Xi$, such that the corresponding connection on a projective module isomorphic to $\Xi$ ceases to  have constant curvature.

\begin{lem}\label{lem:1}\cite{Kang-thesis}
Let $\Xi$ be the $\EE$-$\DD$ projective bimodule given in Section~\ref{sec:prelim}. Let $R$ be the function that gives the Grassmannian connection in \eqref{eq:Gr-conn} and let $Q=\langle R, R \rangle_R^D$ be the corresponding projection in $\DD$. Then the left $Q\DD Q$--right $\DD$ projective bimodule $Q\DD$ is isomorphic to the left $\EE$--right $\DD$ projective bimodule $\Xi$.
\end{lem}

\begin{proof}
Define a map $F$ on $\Xi$ by $F(\xi)=\langle R, \xi\rangle_R^D$. Then $F(\xi)\in Q\DD$ and the inverse $F^{-1}$ is given by $F^{-1}(d)=R\cdot d$ for $d\in Q\DD$. To see this, we compute
\[\begin{split}
F(\xi) &=\langle R,\xi\rangle_R^D =\langle \Id_E\cdot R,\xi\rangle_R^D=\langle \langle R,R\rangle_L^E\cdot R,\xi\rangle_R^D\\
&=\langle R\cdot\langle R,R\rangle_R^D, \xi\rangle_R^D=\langle R\cdot Q, \xi\rangle_R^D=Q^* (\langle \xi,R\rangle_R^D)^*\\
&=Q \langle R,\xi\rangle_R^D\in Q\DD.
\end{split}\]
Also $(F\circ F^{-1})(d)=F(R\cdot d)=\langle R,R\cdot d\rangle_R^D=\langle R,R\rangle_R^D\cdot d=d$ since $d\in Q\DD$. On the other hand $(F^{-1}\circ F)(\xi)=R\cdot F(\xi)=R\cdot \langle R,\xi\rangle_R^D=\langle R,R\rangle_L^E\cdot \xi =\xi$. Thus $F$ is an isomorphism.

Now to see that $F$ preserves the module structure, we define a map $\phi$ on $E$ by $\phi(a)=\langle R,a\cdot R\rangle_R^D$. Then it is straightforward to show that $\phi$ is an injective $\ast$-homomorphism and $\phi(a)\in Q\DD Q$ for all $a\in \EE$. Also for $\Psi\in \EE$, $\Phi\in \DD$ and $\xi, \eta \in \Xi$, we have
\begin{itemize}
\item[(a)] $F(\Psi\cdot \xi)=\phi(\Psi)\ast F(\xi)$,
\item[(b)] $\phi(\langle \xi,\eta\rangle_L^E)=\langle F(\xi), F(\eta)\rangle_L^{Q\DD}$,
\item[(c)] $\langle \xi,\eta\rangle_R^D=\langle F(\xi),F(\eta)\rangle_R^{Q\DD}$,
\item[(d)] $F(\xi\cdot \Phi)= F(\xi)\ast\Phi$,
\end{itemize}
where $\ast$ is the $C^*$-algebra product of $\DD$, $\langle f,g\rangle_L^{Q\DD}=f\ast g^*$ and $\langle f,g\rangle_R^{Q\DD}=f^*\ast g$ for $f,g\in Q\DD$. Therefore $\Xi$ and $Q\DD$ are isomorphic as projective bimodules.
\end{proof}

Using Lemma~\ref{lem:1}, we find the corresponding compatible linear connection $\nabla'$ of the minimizer $\nabla^0$ given in \eqref{eq:min_conn} as follows.

\begin{prop}\label{prop:conn_ex1}
Let $\mathfrak{h}$ be the Heisenberg Lie algebra with basis $\{X,Y,Z\}$ with $[X,Y]=cZ$ given in \eqref{eq:basis-XYZ}.
Let $\nabla^0$ be the compatible connection on $\Xi$ given in \eqref{eq:min_conn} with the constant curvature $\Theta_{\nabla^0}$ given in \eqref{eq:min_curv}. Let $R$ be the function that gives the Grassmannian connection in \eqref{eq:Gr-conn} and let $Q\DD$ be the projective bimodule that is isomorphic to $\Xi$ given in Lemma~\ref{lem:1}. Then we have the following:
\begin{itemize}
\item[(a)] The corresponding compatible connection $\nabla'$ on $Q\DD$ is given by
\[
\nabla'_X(f)=\langle R,\nabla^0_X(R\cdot f)\rangle_R^D
\]
for $X\in \mathfrak{h}$ and $f\in Q\DD$.
\item[(b)] The values of curvature $\Theta'$ of $\nabla'$ lie in $Q\DD Q$ and they are given by
\[\begin{split}
\Theta'_{\nabla'}(X,Y)=0,\quad \Theta'_{\nabla'}(X,Z)=0,\quad\text{and} \quad \Theta'_{\nabla'}(Y,Z)=-\frac{\pi i}{\mu}Q. \\
\end{split}\]
\end{itemize}
\end{prop}
\begin{proof}
First note that the isomorphism $F:\Xi \to Q\DD$ is explicitly given by $F(\xi)=\langle R, \xi\rangle_R^D$ for $\xi\in \Xi$ with the inverse $F^{-1}(d)=R\cdot d$ by Lemma~\ref{lem:1}, where $R\in \Xi$, $\langle R,R\rangle_L^E=Id_E$ and $\langle R, R\rangle_R^D=Q$. Then the corresponding compatible connection $\nabla'$ and the curvature $\Theta_{\nabla'}$ are given by 
\[
\nabla'_X=F\circ \nabla^0_X \circ F^{-1},\quad \text{and}\quad \Theta_{\nabla'}(X,Y)=F\circ \Theta_{\nabla^0}(X,Y)\circ F^{-1}.
\]
Then a straightforward computation shows that 
\[
\nabla'_X(f)=\langle R,\nabla^0_X(R\cdot f)\rangle_R^D,
\]
which gives (a). 

For (b), fix $f\in Q\DD$, we have
\[\begin{split}
&\Theta'_{\nabla'}(X,Y)\cdot f =\nabla'_X(\nabla'_Y(f))-\nabla'_Y(\nabla'_X(f))-\nabla'_{[X,Y]}(f)\\
&=\langle R,\nabla^0_X(R\cdot \nabla'_Y(f))\rangle_R^D-\langle R, \nabla^0_Y(R\cdot \nabla'_X(f))\rangle_R^D-\langle R, \nabla^0_{[X,Y]}(R\cdot f)\rangle_R^D\\
&=\langle R, \nabla^0_X(R\cdot \langle R,\nabla^0_Y(R\cdot f)\rangle_R^D-\langle R,\nabla^0_Y(R\cdot \langle R,\nabla^0_X(R\cdot f)\rangle_R^D)\rangle_R^D  -\langle R,\nabla^0_{[X,Y]}(R\cdot f)\rangle_R^D\\
&=\langle R,\nabla^0_X(\nabla^0_Y(R\cdot f))-\nabla^0_Y(\nabla^0_X(R\cdot f))-\nabla^0_{[X,Y]}(R\cdot f)\rangle_R^D\quad\;\; (\text{since}\;\; \langle R, R\rangle_L^E=\id_E)\\
&=\langle R,\Theta_{\nabla^0}(X,Y)(R\cdot f)\rangle_R^D\\
&=\langle R,\Theta_{\nabla^0}(X,Y)\cdot R\rangle_R^D\cdot f.
\end{split}\]
Thus $\Theta'_{\nabla'}(X,Y)=\langle R,\Theta_{\nabla^0}(X,Y)\cdot R\rangle_R^D.$ Since $\Theta_{\nabla^0}(X,Y)=0$ and $\Theta_{\nabla^0}(X,Z)=0$, we get
\[
\Theta'_{\nabla'}(X,Y)=0\quad\text{and} \quad \Theta'_{\nabla'}(X,Z)=0.
\]
Also the fact that $\langle \cdot, \cdot \rangle_R^D$ is conjugate linear in the second variable implies that 
\[\begin{split}
\Theta'_{\nabla'}(Y,Z)&=\langle R, \Theta_{\nabla^0}(Y,Z)\cdot R \rangle_R^D = \langle R, \frac{\pi i}{\mu} \Id_E \cdot R \rangle_R^D \\
&= -\frac{\pi i}{\mu} \langle R, R\rangle_R^D=-\frac{\pi i}{\mu} Q.
\end{split}\]

To see that $\Theta'_{\nabla'}(X,Y)\in Q\DD Q$, we compute
\[\begin{split}
\Theta'_{\nabla'}(X,Y)&=\langle R,\Theta_{\nabla^0}(X,Y)\cdot R\rangle_R^D\\
&=\langle \langle R,R\rangle_L^E\cdot R, \Theta_{\nabla^0}(X,Y)\cdot \langle R, R\rangle_L^E\cdot R\rangle_R^D\\
&=\langle R\cdot \langle R, R,\rangle_R^D,\Theta_{\nabla^0}(X,Y)\cdot R\cdot \langle R, R\rangle_R^D\rangle_R^D\\
&=\langle R, R\rangle_R^D \langle R,\Theta_{\nabla^0}(X,Y)\cdot R\rangle_R^D\langle R, R\rangle_R^D \\
&=Q\langle R,\Theta_{\nabla^0}(X,Y)\cdot R\rangle_R^D Q \in Q\DD Q.
\end{split}\]

\end{proof}

Now we show that $\nabla'$ is not a critical point of $\YM$:

\begin{prop}\label{prop:conn_ex1-not-cp}
The compatible connection $\nabla'$ on $Q\DD$ given in Proposition~\ref{prop:conn_ex1}(a)  is not a critical point of the Yang-Mills functional $\YM$ on $\DD$.

\end{prop}

\begin{proof}
If $\nabla'$ is a critical point of $\YM$, then $\nabla'$ should satisfy equations (1)--(3) in \eqref{eq:critical_point}, in particular (2), which is

\begin{equation}
[\nabla'_X,\Theta'_{\nabla'}(Y,X)] + [\nabla_Z, \Theta'_{\nabla'}(Y,Z)]=0,
\end{equation}

We show that $\nabla'$ does not satisfy the above equation. Since $\Theta'_{\nabla'}(X,Y)=\Theta'_{\nabla'}(Z,X)=0$ and $\Theta'_{\nabla'}(Y,Z)=-\frac{\pi i}{\mu} Q$, we get
\[\begin{split}
&\Big([\nabla'_X,\Theta'_{\nabla'}(Y,X)] + [\nabla'_Z,\Theta'_{\nabla'}(Y,Z)]\Big)\cdot f \\
&=\nabla'_Z(-\frac{\pi i}{\mu} Q\cdot f) + \frac{\pi i}{\mu} Q(\nabla'_Z(f)) \\
&=\langle R, \nabla^0_Z\big(R\cdot \frac{-\pi i}{\mu}Q\cdot f \big)\rangle_R^D +\frac{\pi i}{\mu} Q\langle R, \nabla^0_Z(R\cdot f)\rangle_R^D\\
\end{split}\]
But $Q=\langle R, R\rangle_R^D$, thus we get
\[
\begin{split}
&\Big([\nabla'_X,\Theta'_{\nabla'}(Y,X)] + [\nabla'_Z,\Theta'_{\nabla'}(Y,Z)]\Big)\cdot f \\
&=\langle R,\nabla^0_Z\big(R\cdot (\frac{-\pi i}{\mu} \langle R, R\rangle_R^D\cdot f)\big)\rangle_R^D + \frac{\pi i}{\mu}\langle R, R\rangle_R^D\langle R,\nabla^0_Z(R\cdot f)\rangle_R^D\\
&=\langle R,(\nabla^0_Z\big((\frac{-\pi i}{\mu} \langle R, R\rangle_L^E\cdot R) \cdot f\big)\rangle_R^D + \frac{\pi i}{\mu} \langle R, R\cdot \langle R, \nabla^0_Z(R\cdot f)\rangle_R^D\rangle_R^D\\
&= \frac{\pi i}{\mu} \langle R,\nabla^0_Z(R\cdot f)\rangle_R^D+\frac{\pi i}{\mu}\langle R, \nabla^0_Z(R\cdot f)\rangle_R^D\\
&= \frac{2\pi i}{\mu} \langle R, \nabla^0_Z(R\cdot f)\rangle_R^D\ne 0.
\end{split}\]
Therefore $\nabla'$ is not a critical point of $\YM$.
\end{proof}

Proposition~\ref{prop:conn_ex1} and Proposition~\ref{prop:conn_ex1-not-cp} clearly show that the geometric structure on algebraically and topologically isomorphic projective modules can be quite different. In particular, $\nabla'$ on $Q\DD$ has non-constant curvature and does not give a critical point of $\YM$ even though the corresponding connection $\nabla^0$ on the isomorphic projective module $\Xi$ has constant curvature and does give a critical point of $\YM$.

\subsection{Connections with constant curvature}

To this point, we have seen only one compatible connection on $\Xi$ with constant curvature on which $\YM$ attains  a local minimum and that is a critical point of $\YM$, constructed in  \cite{Lee}. One might ask if $\YM$ attains its minimum and has a critical point at every compatible connection on $\Xi$ with constant curvature of $\YM$ for the quantm Heisenberg manifolds $\DD$. Here we give an example of a compatible connection with constant curvature on $\Xi$ which is neither a minimizer nor a critical point of $\YM$, which shows that our question must be answered in the negative. 

\begin{thm}\label{thm:const_conn2}
Let $\Xi$ be the left $\EE$ and right $\DD$ projective bimodule discussed in the previous sections, and let $\mathfrak{h}$ be the Heisenberg Lie algebra with basis $\{X,Y,Z\}$ with $[X,Y]=cZ$ given in \eqref{eq:basis-XYZ}.
 Suppose that $\mu\ne 0$ and $\nu\ne 0$. Define a linear map $\nabla^1:\Xi\to \Xi\otimes \mathfrak{h}^*$ by
\[\begin{split}
(\nabla^1_Xf)(x,y)&=-\frac{\partial f}{\partial y}(x,y)+(\frac{\pi c i}{2\mu}x^2-\nu i x+\mu i y)f(x,y)\\
(\nabla^1_Yf)(x,y)&=-\frac{\partial f}{\partial x}(x,y)\\
(\nabla^1_Zf)(x,y)&=\frac{\pi i x}{\mu}f(x,y)
\end{split}\]
Then $\nabla^1$ is a compatible linear connection with constant curvature
\[
\Theta_{\nabla^1}(X,Y)=\nu i \Id_E,\quad \Theta_{\nabla^1}(X,Z)=0,\quad \Theta_{\nabla^1}(Y,Z)=\frac{\pi i}{\mu}\Id_E,
\]
where $\Id_E(x,y,p)=\delta_0(p)$.
\end{thm}
\begin{proof}
Since $\nabla^1_Y$ and $\nabla^1_Z$ are the same as $\nabla_Y$ and $\nabla_Z$ given in \eqref{eq:min_conn}, to show that $\nabla^1$ is a compatible linear connection, we only need to check that $\nabla^1_X$ satisfies
\[\begin{split}
&\nabla^1_X(f\cdot \Phi)=\nabla^1(f)\cdot \Phi+f\cdot \delta_X(\Phi),\quad{and}\\
&\langle \nabla^1_X(f),g\rangle_R^D+\langle f,\nabla^1_X(g)\rangle_R^D=\delta_X(\langle f,g\rangle_R^D).
\end{split}\]
for all $f\in \Xi$, $\Phi\in \DD$ and $X\in \mathfrak{h}$. We compute
\[
\nabla^1_X(f\cdot \Phi)(x,y) =-\frac{\partial}{\partial y}(f\cdot \Phi)(x,y)+(\frac{\pi ci}{2\mu}x^2-\nu ix+\mu iy)(f\cdot \Phi)(x,y)\
\]
\[
=-\frac{\partial}{\partial y}\Big(\sum_q f(x+2q\mu,y+2q\nu)\overline{\Phi}(x+2q\mu,y+2q\nu,q)\Big)+(\frac{\pi ci}{2\mu}x^2-\nu ix+\mu iy)
\]
\[
\quad \times \Big(\sum_q f(x+2q\mu,y+2q\nu)\overline{\Phi}(x+2q\mu,y+2q\nu,q)\Big)
\]
\[
=-\sum_q \frac{\partial f}{\partial y}(x+2q\mu,y+2q\nu)\overline{\Phi}(x+2q\mu,y+2q\nu,q) -\sum_q \big(f(x+2q\mu,y+2q\nu)
\]
\[
\quad\quad \times\frac{\partial \overline{\Phi}}{\partial y}(x+2q\mu,y+2q\nu,q)\big)+ (\frac{\pi ci}{2\mu}x^2-\nu ix+\mu iy)\sum_q f(x+2q\mu,y+2q\nu)\overline{\Phi}(x+2q\mu,y+2q\nu,q)\\
\]
On the other hand,
\[\begin{split}
&(\nabla^1(f)\cdot \Phi)(x,y)+(f\cdot \delta_X(\Phi))(x,y)\\
&=\sum_q (\nabla^1_Xf)(x+2q\mu,y+2q\nu)\overline{\Phi}(x+2q\mu,y+2q\nu,q)+\sum_q f(x+2q\mu,y+2q\nu)\\
&\quad\times(\overline{\delta_X\Phi})(x+2q\mu,y+2q\nu,q)
\end{split}\]
\[
=\sum_q \Big(-\frac{\partial f}{\partial y}(x+2q\mu,y+2q\nu)+\big(\frac{\pi ci}{2\mu}(x+2q\mu)^2-\nu i(x+2q\mu)+\mu i(y+2q\nu)\big)
\]
\[
\quad \times f(x+2q\mu,y+2q\nu)\Big)\overline{\Phi}(x+2q\mu,y+2q\nu,q) + \sum_q f(x+2q\mu,y+2q\nu)
\]
\[
\quad \times\Big(-\frac{\partial \overline{\Phi}}{\partial y}(x+2q\mu,y+2q\nu,q) +\overline{2\pi icq}(x+2q\mu-q\mu)\overline{\Phi}(x+2q\mu,y+2q\nu,q)\Big)
\]
\[
=-\sum_q \frac{\partial f}{\partial y}(x+2q\mu,y+2q\nu)\overline{\Phi}(x+2q\mu,y+2q\nu,q) -\sum_q \big(f(x+2q\mu,y+2q\nu)
\]
\[
\quad \times\frac{\partial \overline{\Phi}}{\partial y}(x+2q\mu,y+2q\nu,q)\big) +\sum_q \Big(f(x+2q\mu,y+2q\nu)\overline{\Phi}(x+2q\mu,y+2q\nu,q)
\]
\[
\quad \times \Big( \frac{\pi ci}{2\mu}(x+2q\mu)^2-\nu i(x+2q\mu)+\mu i(y+2q\nu)-2\pi icq(x+q\mu)\Big)\Big)
\]
\[
=-\sum_q \frac{\partial f}{\partial y}(x+2q\mu,y+2q\nu)\overline{\Phi}(x+2q\mu,y+2q\nu,q) -\sum_q \big(f(x+2q\mu,y+2q\nu)
\]
\[
\quad \times\frac{\partial \overline{\Phi}}{\partial y}(x+2q\mu,y+2q\nu,q)\big) +\sum_q \Big(f(x+2q\mu,y+2q\nu)\overline{\Phi}(x+2q\mu,y+2q\nu,q)
\]
\[
\quad \times(\frac{\pi ci}{2\mu}x^2-\nu ix+\mu iy)\Big) = \nabla^1_X(f\cdot \Phi)(x,y).
\]
Thus
\[
\nabla^1_X(f\cdot \Phi)(x,y)=(\nabla^1(f)\cdot \Phi)(x,y)+(f\cdot \delta_X(\Phi))(x,y).
\]

For compatibility, first note that
\[\begin{split}
&\delta_X(\langle f, g\rangle_R^D)(x,y,p)\\
&=2 \pi icp(x-p\mu)(\langle f, g\rangle_R^D)(x,y,p)-\frac{\partial}{\partial y}(\langle f, g\rangle_R^D)(x,y,p)\\
&=\sum_k 2\pi icp (x-p\mu+k)\overline{e}(ckp(y-p\nu)) f(x+k,y)\overline{g}(x-2p\mu+k,y-2p\nu).\\
&\quad -\sum_k \overline{e}(ckp(y-p\nu))\frac{\partial f}{\partial y}(x+k,y)\overline{g}(x-2p\mu+k,y-2p\nu)\\
&\quad - \sum_k \overline{e}(ckp(y-p\nu)) f(x+k,y)\frac{\partial \overline{g}}{\partial y}(x-2p\mu+k,y-2p\nu).\\
\end{split}\]
 Also we compute
\[\begin{split}
&\langle \nabla^1_X(f),g\rangle_R^D(x,y,p)+\langle f,\nabla^1_X(g)\rangle(x,y,p)\\
&=\sum_k\overline{e}(ckp(y-p\nu))(\nabla^1_Xf)(x+k,y)\overline{g}(x-2p\mu+k,y-2p\nu)\\
&\quad +\sum_k\overline{e}(ckp(y-p\nu))f(x+k,y)(\overline{\nabla^1_Xg})(x-2p\mu+k,y-2p\nu)\\
&=\sum_k\overline{e}(ckp(y-p\nu))\Big(-\frac{\partial f}{\partial y}(x+k, y)+\big(\frac{\pi ci}{2\mu}(x+k)^2-\nu i(x+k)+\mu iy\big) f(x,y)\Big)\\
&\quad \times \overline{g}(x-2p\mu+k,y-2p\nu)\\
&+\sum_k \overline{e}(ckp(y-p\nu))f(x+k,y)\Big(-\frac{\partial \overline{g}}{\partial y}(x-2p\mu+k,y-2p\nu)\\
&\quad +\big(\overline{\frac{\pi c i}{2\mu}(x-2p\mu+k)^2-\nu i(x-2p\mu+k)+\mu i(y-2p\nu)}\big)\overline{g}(x-2p\mu+k, y-2p\nu)\Big)\\
\end{split}\]
\[\begin{split}
&=-\sum_k \overline{e}(ckp(y-p\nu))\frac{\partial f}{\partial y}(x+k,y)\overline{g}(x-2p\mu+k,y-2p\nu)\\
&\quad - \sum_k \overline{e}(ckp(y-p\nu)) f(x+k,y)\frac{\partial \overline{g}}{\partial y}(x-2p\mu+k,y-2p\nu)\\
&\quad +\sum_k \overline{e}(ckp(y-p\nu)) f(x+k,y)\overline{g}(x-2p\mu+k,y-2p\nu)\Big(\frac{\pi ci}{2\mu}(x+k)^2-\nu i(x+k)+\mu i y\\
&\quad -\frac{\pi ci}{2\mu}(x-2p\mu +k)^2+\nu i(x-2p\mu+k)-\mu i(y-2p\nu)\Big)\\
&=\sum_k \overline{e}(ckp(y-p\nu))\frac{\partial f}{\partial y}(x+k,y)\overline{g}(x-2p\mu+k,y-2p\nu)\\
&\quad - \sum_k \overline{e}(ckp(y-p\nu)) f(x+k,y)\frac{\partial \overline{g}}{\partial y}(x-2p\mu+k,y-2p\nu)\\
&\quad +\sum_k \overline{e}(ckp(y-p\nu)) f(x+k,y)\overline{g}(x-2p\mu+k,y-2p\nu)\big(2\pi icp(x+k-p\mu)\big)
\end{split}\]
Thus $\langle \nabla^1_X(f),g\rangle_R^D+\langle f,\nabla^1_X(g)\rangle_R^D=\delta_X(\langle f,g\rangle_R^D)$. Hence $\nabla^1$ is a compatible connection on $\Xi$.

Now we compute the curvature $\Theta_{\nabla^1}$ as follows. Fix $f\in \Xi$ and compute
\[\begin{split}
&(\Theta_{\nabla^1}(X,Y)\cdot f)(x,y)=\nabla^1_X(\nabla^1_Yf)(x,y)-\nabla^1_Y(\nabla^1_Xf)(x,y)-(\nabla^1_{[X,Y]}f)(x,y)\\
&=-\frac{\partial}{\partial y}(\nabla^1_Yf)(x,y)+\big(\frac{\pi ci}{2\mu}x^2-\nu ix+\mu iy\big)(\nabla^1_Yf)(x,y)+\frac{\partial}{\partial x}(\nabla^1_Xf)(x,y)-c(\nabla^1_Zf)(x,y)\\
&=-\frac{\partial}{\partial y}\Big(-\frac{\partial f}{\partial x}(x,y)\Big)+\big(\frac{\pi ci}{2\mu}x^2-\nu ix+\mu iy\big)\Big(-\frac{\partial f}{\partial x}(x,y)\Big)+\frac{\partial}{\partial x}\Big(-\frac{\partial f}{\partial y}(x,y)\\
&\quad +\big(\frac{\pi ci}{2\mu}x^2-\nu ix+\mu iy\big)f(x,y)\Big)-c\frac{\pi ix}{\mu}f(x,y)\\
&=\frac{\partial^2 f}{\partial y\partial x}(x,y)-\big(\frac{\pi ci}{2\mu}x^2-\nu ix+\mu iy\big)\frac{\partial f}{\partial x}(x,y)-\frac{\partial^2 f}{\partial y\partial x}(x,y)+\big(\frac{\pi ci}{2\mu}2x-\nu i\big)f(x,y)\\
&\quad + \big(\frac{\pi ci}{2\mu}x^2-\nu ix+\mu iy\big)\frac{\partial f}{\partial x}(x,y)-\frac{\pi icx}{\mu}f(x,y)\\
&=-\nu i f(x,y).
\end{split}\]
Thus $\Theta_{\nabla^1}(X,Y)=\nu i \id_E$. Also we compute
\[\begin{split}
&(\Theta_{\nabla^1}(X,Z)\cdot f)(x,y)=\nabla^1_X(\nabla^1_Zf)(x,y)-\nabla^1_Z(\nabla^1_Xf)(x,y)\\
&=-\frac{\partial}{\partial y}(\nabla^1_Zf)(x,y)+\big(\frac{\pi ci}{2\mu}x^2-\nu ix+\mu iy\big)(\nabla^1_Zf)(x,y)-\frac{\pi ix}{\mu}(\nabla^1_Xf)(x,y)\\
&=-\frac{\partial}{\partial y}\Big(\frac{\pi ix}{\mu}f(x,y)\Big)+\big(\frac{\pi ci}{2\mu}x^2-\nu ix+\mu iy\big)\big(\frac{\pi ix}{\mu}f(x,y)\Big)-\frac{\pi ix}{\mu}\Big(-\frac{\partial f}{\partial y}(x,y)\\
&\quad+\big(\frac{\pi ci}{2\mu}x^2-\nu ix+\mu iy\big)f(x,y)\Big)\\
&=0.
\end{split}\]
Thus $\Theta_{\nabla^1}(X,Z)=0$. Since $\nabla^1_Y$ and $\nabla^1_Z$ are the same as $\nabla^0_Y$ and $\nabla^0_Z$ given in \eqref{eq:min_conn}, $\Theta_{\nabla^1}(Y,Z)=\Theta_{\nabla^0}(Y,Z)=\frac{\pi i}{\mu}\Id_E$, which completes the proof.
\end{proof}

To prove the next proposition, we need the following lemma that shows how the compatible connection $\nabla^0$ given in \eqref{eq:min_conn} acts on the multiplication-type element of $\EE$ introduced in \cite{Kang1}.

\begin{lem}\label{lem:nabla_G}
Let $\nabla^0$ be the compatible connection on $\Xi$ given in \eqref{eq:min_conn}.  Let $\mathbb{G}$ be a skew-symmetric multiplication-type element of $\EE $. i.e.  $\mathbb{G}^*=-\mathbb{G}$ and $\mathbb{G}(x,y,p)=G(x,y)\delta_0(p)$, where $G$ is a skew-symmetric\footnote{According to Lemma~6 of \cite{Kang1}, $\mathbb{G}$ is skew symmetric if and only if the corresponding function $G$ is skew-symmetric, i.e. $\overline{G}(x,y)=-G(x,y)$.} differentiable function on $\RR\times \TT$. Then for $\xi\in \Xi$, we have
\begin{equation}\label{eq:conn_XG}
([\nabla^0_X, \mathbb{G}]\cdot \xi)=\frac{\partial G}{\partial y}(x,y)\xi(x,y),
\end{equation}

\begin{equation}\label{eq:conn_YG}
([\nabla^0_Y, \mathbb{G}]\cdot \xi)(x,y)=\frac{\partial G}{\partial x}(x,y)\xi(x,y),
\end{equation}

\begin{equation}\label{eq:conn_ZG}
([\nabla^0_Z, \mathbb{G}]\cdot \xi)(x,y)=0.
\end{equation}

\end{lem}
\begin{proof}
Fix $\xi\in \Xi$. Then Proposition~7 of \cite{Kang1} implies that $(\mathbb{G}\cdot \xi)(x,y)=-G(x,y)\xi(x,y)$ for $\xi\in \Xi$ since $\mathbb{G}$ is a multiplication-type element of $\EE$.
We compute
\[\begin{split}
([\nabla^0_X, \mathbb{G}]\cdot \xi)(x,y) &=\nabla^0_X(\mathbb{G}\cdot \xi)(x,y)-(\mathbb{G}\cdot \nabla^0_X(\xi))(x,y)\\
&=-\frac{\partial}{\partial y}(\mathbb{G}\cdot \xi)(x,y)+\frac{\pi c i}{2\mu}x^2(\mathbb{G}\cdot \xi)(x,y) -(\mathbb{G}\cdot \nabla^0_X(\xi))(x,y)\\
&=-\frac{\partial}{\partial y}\big(-G(x,y)\xi(x,y)\big)+\frac{\pi c i}{2\mu} x^2 \big(-G(x,y)\xi(x,y)\big)\\
&\quad\quad\quad+G(x,y)\big(-\frac{\partial \xi}{\partial y}(x,y)+\frac{\pi c i}{2\mu}x^2\xi(x,y)\big)\\
&=\frac{\partial G}{\partial y}(x,y)\xi(x,y),
\end{split}\]
which gives equation \eqref{eq:conn_XG}. Also we compute
\[\begin{split}
([\nabla^0_Y, \mathbb{G}]\cdot \xi)(x,y)&=\nabla^0_Y(\mathbb{G}\cdot \xi)(x,y)-(\mathbb{G}\cdot \nabla^0_Y(\xi))(x,y)\\
&=-\frac{\partial}{\partial x}(\mathbb{G}\cdot \xi)(x,y)+G(x,y)(-\frac{\partial \xi}{\partial x}(x,y))\\
&=\frac{\partial G}{\partial x}(x,y) \xi(x,y),
\end{split}\]
which gives \eqref{eq:conn_YG}. To see \eqref{eq:conn_ZG},
\[\begin{split}
([\nabla^0_Z, \mathbb{G}]\cdot \xi)(x,y)&=\nabla^0_Z(\mathbb{G}\cdot \xi)(x,y)-(\mathbb{G}\cdot \nabla^0_Z(\xi))(x,y)\\
&=\frac{\pi i x}{\mu}(\mathbb{G}\cdot \xi)(x,y)+G(x,y)(\frac{\pi i x}{\mu}\xi(x,y))\\
&=\frac{\pi i x}{\mu}(-G(x,y)\xi(x,y))+G(x,y)(\frac{\pi i x}{\mu}\xi(x,y)) = 0
\end{split}\]

\end{proof}

\begin{prop}\label{prop:nabla-1-not-cp}
The compatible connection $\nabla^1$ with constant curvature given in Theorem~\ref{thm:const_conn2} is  neither a critical point nor a minimizer of $\YM$.
\end{prop}
\begin{proof}
We will first show that $\nabla^1$ does not satisfy (3) of \eqref{eq:critical_point} and hence $\nabla^1$ is not a critical point of $\YM$.
Note first that any curvature $\Theta_{\nabla'}$ of any compatible connection $\nabla'$ is a skew-symmetric element of $\EE$.
Since $\Theta_{\nabla^1}(X,Y)$ is a pure imaginary constant multiple of the identity element $\Id_E$ of $\EE$ for $X,Y\in\mathfrak{h}$, the curvature $\Theta_{\nabla^1}(X,Y)$ is a skew-symetric multiplication-type element of $\EE$.  Since $\Theta_{\nabla^1}(X,Y)=\nu i \Id_E$,  $\Theta_{\nabla^1}(X,Z)=0$, $\Theta_{\nabla^1}(Y,Z)=\frac{\pi i}{\mu}\Id_E$, and $\nabla^1_Y=\nabla_Y$, Lemma~\ref{lem:nabla_G} implies that
\[\begin{split}
&[\nabla^1_X,\Theta_{\nabla^1}(Z,X)]+[\nabla^1_Y, \Theta_{\nabla^1}(Z,Y)]-c\Theta_{\nabla^1}(X,Y)\\
&=0+[\nabla_Y, -\frac{\pi i}{\mu}\Id_E]-c\nu i\Id_E\\
&=-c\nu i\Id_E\ne 0.
\end{split}\]
Thus $\nabla^1$ does not satisfy (3) of \eqref{eq:critical_point}, and hence $\nabla^1$ is not a critical point of $\YM$.

To see that $\nabla^1$ does not even give a local minimum of $\YM$, we compare the value of $\YM(\nabla')$ to that of $\YM(\nabla^0)$, where $\nabla^0$ is the connection given in \eqref{eq:min_conn}. In fact, we have
\[\begin{split}
\YM(\nabla^1)&=-\tau_E(\{\Theta_{\nabla^1}, \Theta_{\nabla^1}\})=-\int_0^{2\mu} \int_0^1 (-\nu^2-\frac{\pi^2}{\mu^2})\, dy\, dx\\
&=2\mu\nu^2+\frac{2\pi^2}{\mu} >\frac{2\pi^2}{\mu}=\YM(\nabla^0).
\end{split}\]
Hence $\nabla^1$ is not a minimizer of $\YM$.

\end{proof}

Note that the proof of Proposition~\ref{prop:nabla-1-not-cp} gets a lot simpler once we characterize critical points and minimizers of $\YM$ with constant curvature in the next section. See Theorem~\ref{thm:critical-ioi}, Theorem~\ref{prop:min-cp}, and remarks after.

\section{Yang-Mills connections with constant curvature}\label{sec:YM-conn}

In this section, we investigate Yang-Mills connections on $\DD$ with constant curvature.
We first study how to identify critical points of the Yang-Mills functional $\YM$.
\begin{prop}\label{prop:critical}
Let $\Xi$ be the left $\EE$ --  right $\DD$ projective bimodule described in Section~\ref{sec:prelim} and let $\mathfrak{h}$ be the Heisenberg Lie algebra with basis $\{X,Y,Z\}$ with $[X,Y]=cZ$ given in \eqref{eq:basis-XYZ}. Suppose $\nabla$ is a compatible connection on $\Xi$ with curvature $\Theta_\nabla$.  If $\nabla$ is a critical point of the Yang-Mills functional $\YM$ given in \eqref{eq:YM-functional}, then $\Theta_\nabla(X,Y)=0$. 
\end{prop}

\begin{proof}
If $\nabla$ is a critical point of $\YM$, then $\nabla$ satisfies (1), (2) and (3) of \eqref{eq:critical_point}. By interchanging $X$ and $Y$ in (3), we obtain
\[
(3)':\; [\nabla_Y,\Theta_\nabla(Z,Y)]+[\nabla_X, \Theta_\nabla(Z,X)]-c\Theta_\nabla(Y,X)=0.
\]
Then by substracting (3) of \eqref{eq:critical_point} from $(3)'$, we obtain
\[
-c\Theta_\nabla(Y,X)+c\Theta_\nabla(X,Y)=0.
\]
Since $\Theta_\nabla(X,Y)=-\Theta_\nabla(Y,X)$ and $c>0$, we get $\Theta_\nabla(X,Y)=0$, which proves the desired result.
\end{proof}

We remark that the converse of Proposition~\ref{prop:critical} is not necessarily true in general. 
However, for $\nabla$ having constant curvature, we obtain the following result.

\begin{thm}\label{thm:critical-ioi}
Let $\Xi$ and $\mathfrak{h}$ be as in Proposition~\ref{prop:critical}. Suppose a compatible connection $\nabla$ has  constant curvature $\Theta_\nabla$.  Then $\nabla$ is a critical point of the Yang-Mills functional $\YM$ given in \eqref{eq:YM-functional} if and only if $\Theta_\nabla(X,Y)=0$. 
\end{thm}
\begin{proof}
First note that since $\nabla$ is a compatible connection, we can write $\nabla=\nabla^0+\mathbb{H}$, where $\nabla^0$ is the compatible connection given in \eqref{eq:min_conn} and $\mathbb{H}$ is a linear map from $\mathfrak{h}$ into the set of skew-symmetric elements of $\EE$. i.e. $(\mathbb{H}_X)^*=-\mathbb{H}_X$ for all $X\in \mathfrak{h}$.

Now suppose that $\nabla$ is a critical point. Then $\nabla$ satisfies (1),(2) and (3) in \eqref{eq:critical_point}. Since $\nabla$ has constant curvature, we can write $\Theta_\nabla(X,Y)=a_1 i \Id_E$, $\Theta_\nabla(X,Z)=a_2 i \Id_E$ and $\Theta_\nabla(Y,Z)=a_3 i\Id_E$, where $a_1, a_2, a_3\in \R$. Then by Lemma~\ref{lem:nabla_G} we get
\[
[\nabla^0_X, a_j i \Id_E]=0, \quad [\nabla^0_Y, a_j i \Id_E]=0, \quad \text{and} \quad [\nabla^0_Z, a_j i \Id_E]=0.
\]
for $j=1,2,3$, and hence
\[
[\nabla_X, a_j i\id_E]=[\nabla^0_X+\mathbb{H}, a_j i \Id_E]=[\nabla^0_X, a_j i \Id_E]+ [\mathbb{H}, a_j i \Id_E]=0+0=0
\]
for $j=1,2,3$. Similarly then we have
\[
[\nabla_Y, a_j i\Id_E]=0 \quad \text{and} \quad [\nabla_Z, a_j i\Id_E]=0
\]
for $j=1,2,3$. Thus (3) of \eqref{eq:critical_point} gives $\Theta_\nabla(X,Y)=0$.

On the other hand, if $\Theta_\nabla(X,Y)=0$, then one can immediately see that $\nabla$ satisfies (1), (2) and (3) of \eqref{eq:critical_point} since $\Theta_\nabla(X,Z)$ and $\Theta_\nabla(Y,Z)$ are constant. Hence $\nabla$ is a critical point of $\YM$.

\end{proof}

One can now immediately see that  the compatible connection $\nabla^1$ with constant curvature given in Theorem~\ref{thm:const_conn2} is not a critical point of $\YM$ by Theorem~\ref{thm:critical-ioi}  since $\Theta_{\nabla^1}(X,Y)=\nu i \ne 0$.

 The following proposition shows that a minimizing connection $\nabla$ with constant curvature should have a certain form of constant curvature, and thus $\nabla$ gives a critical point of $\YM$.

\begin{thm}\label{prop:min-cp}

Let $\Xi$ be the left $\EE$ and right $\DD$ projective bimodule described in Section~\ref{sec:prelim}.
Let $\nabla$ be a compatible connection on $\Xi$ over $\DD$ with constant curvature. Then $\nabla$ is a minimizer of $\YM$ subject to the constant curvature constraint, in the sense that $\YM(\nabla) \le \YM(\nabla')$ for a compatible connection $\nabla'$ with constant curvature if and only if  the curvature $\Theta_\nabla$ is the same as the curvature $\Theta_{\nabla^0}$, where $\nabla^0$ is the compatible connection given in \eqref{eq:min_conn}, i.e.

\begin{equation}\label{eq:curv-min}
\Theta_{\nabla}(X,Y)=0,\;\;\Theta_{\nabla}(X,Z)=0,\;\;\Theta_{\nabla}(Y,Z)=\frac{\pi i}{\mu}\Id_E,
\end{equation}

\end{thm}

\begin{proof}
 Suppose that  $\nabla$ is a minimizer of $\YM$ subject to the constant curvature constraint, in the sense that $\YM(\nabla) \le \YM(\nabla')$ for a compatible connection $\nabla'$ with constant curvature.   In particular, $\YM(\nabla)\le \YM(\nabla^0)$, where $\nabla^0$ is the compatible connection on $\Xi$ given in \eqref{eq:min_conn}.   Since $\nabla$ is a compatible connection on $\Xi$, we have $\nabla=\nabla^0+\mathbb{H}$ for a skew-symmetric element $\mathbb{H}\in \EE$.
 Then the curvature of $\nabla$ is given by
\[\begin{split}
&\Theta_{\nabla}(X,Y)=\Theta_{\nabla^0}(X,Y)+\Psi(X,Y)=\Psi(X,Y),\\
&\Theta_{\nabla}(X,Z)=\Theta_{\nabla^0}(X,Z)+\Psi(X,Z)=\Psi(X,Z),\\
&\Theta_{\nabla}(Y,Z)=\Theta_{\nabla^0}(Y,Z)+\Psi(Y,Z)=\frac{\pi i}{\mu} \Id_E +\Psi(Y,Z),
\end{split}\]
where
\[\begin{split}
&\Psi(X,Y)=[\nabla^0_X,\mathbb{H}_Y]-[\nabla^0_Y,\mathbb{H}_X]+[\mathbb{H}_X,\mathbb{H}_Y]-\mathbb{H}_{[X,Y]},\\
&\Psi(X,Z)=[\nabla^0_X, \mathbb{H}_Z]-[\nabla^0_Z, \mathbb{H}_X]+[\mathbb{H}_X,\mathbb{H}_Z],\\
&\Psi(Y,Z)=[\nabla^0_Y, \mathbb{H}_Z]-[\nabla^0_Z, \mathbb{H}_Y]+[\mathbb{H}_Y,\mathbb{H}_Z].
\end{split}\]
Since we assume that the curvature of $\nabla$ is constant, we have
\[
\Psi(X,Y)=a_1 i \Id_E, \quad \Psi(X,Z)=a_2 i \Id_E, \quad \Psi(Y,Z)=a_3 i \Id_E
\]
for some $a_1, a_2, a_3 \in \R$.

Note that $\YM(\nabla^0)=-\tau_E(\{\Theta_{\nabla^0}, \Theta_{\nabla^0}\})=-\tau_E((\frac{\pi i}{\mu} \Id_E)^2)=-\int_0^{2\mu}\int_0^1 (-\frac{\pi^2}{\mu^2})dy\, dx=\frac{2\pi^2}{\mu}$. 
Also note that $\tau_E(\Psi(Y,Z))=\tau_E([\nabla^0_Y, \mathbb{H}_Z]-[\nabla^0_Z, \mathbb{H}_Y]+[\mathbb{H}_Y
,\mathbb{H}_Z])=0$ by Lemma~2.2 of \cite{CR}.

Then we have
\begin{equation}\label{eq:YM-min}
\begin{split}
\YM(\nabla)&=-\tau_E(\{\Theta_\nabla, \Theta_\nabla\})=-\tau_E((\Theta_{\nabla}(X,Y))^2+(\Theta_{\nabla}(X,Z))^2+(\Theta_{\nabla}(Y,Z))^2)\\
&= -\tau_E((\Psi(X,Y))^2+(\Psi(X,Z))^2 + ((\frac{\pi i}{\mu}\Id_E+\Psi(Y,Z))^2)\\
&=\frac{2\pi^2}{\mu}-\tau_E\big(\frac{2\pi i}{\mu}\Psi(Y,Z)\big)  -\tau_E\big((\Psi(X,Y))^2+(\Psi(X,Z))^2+(\Psi(Y,Z))^2\big)\\
&=\frac{2\pi^2}{\mu}-\tau_E(-a_1^2 \Id_E-a_2^2 \Id_E -a_3^2 \Id_E)\\
&=\frac{2\pi^2}{\mu}+2\mu(a_1^2+a_2^2+a_3^2) \; \le \frac{2\pi^2}{\mu}=\YM(\nabla^0).
\end{split}
\end{equation}
Thus we should have $a_1^2+a_2^2+a_3^2=0$ since $\mu>0$, and hence $a_1=a_2=a_3=0$.
This implies that $\Psi(X,Y)=\Psi(X,Z)=\Psi(Y,Z)=0$. Therefore, the curvature of $\nabla$ is given by
\[
\Theta_{\nabla}(X,Y)=0, \quad \Theta_{\nabla}(X,Z)=0, \quad \Theta_{\nabla}(Y,Z)=\frac{\pi i }{\mu}\Id_E.
\]

Conversely, suppose that $\nabla$ has constant curvature of the form given by \eqref{eq:curv-min}. To show that $\nabla$ is a minimizer of $\YM$ subject to the constant curvature constraint,  consider $\nabla'=\nabla+\mathbb{F}$ with constant curvature  $\Theta_{\nabla'}$ for a skew-symmetric element $\mathbb{F}\in \EE$. 
Then
\[\begin{split}
&\Theta_{\nabla'}(X,Y)=\Theta_{\nabla}(X,Y)+\Psi(X,Y)=\Psi(X,Y),\\
&\Theta_{\nabla'}(X,Z)=\Theta_{\nabla}(X,Z)+\Psi(X,Z)=\Psi(X,Z),\\
&\Theta_{\nabla'}(Y,Z)=\Theta_{\nabla}(Y,Z)+\Psi(Y,Z)=\frac{\pi i}{\mu} \Id_E +\Psi(Y,Z),
\end{split}\]
where
\[\begin{split}
&\Psi(X,Y)=[\nabla_X,\mathbb{F}_Y]-[\nabla_Y,\mathbb{F}_X]+[\mathbb{F}_X,\mathbb{F}_Y]-\mathbb{F}_{[X,Y]},\\
&\Psi(X,Z)=[\nabla_X, \mathbb{F}_Z]-[\nabla_Z, \mathbb{F}_X]+[\mathbb{F}_X,\mathbb{F}_Z],\\
&\Psi(Y,Z)=[\nabla_Y, \mathbb{F}_Z]-[\nabla_Z, \mathbb{F}_Y]+[\mathbb{F}_Y,\mathbb{F}_Z].
\end{split}\]
Since $\nabla'$ has constant curvature, we should have
\[\begin{split}
&\Theta_{\nabla'}(X,Y)=\Psi(X,Y)=b_1 i \Id_E,\\
& \Theta_{\nabla'}(X,Z)=\Psi(X,Z)=b_2 i \Id_E,\\
& \Theta_{\nabla'}(Y,Z)=\frac{\pi i}{\mu} \Id_E +\Psi(Y,Z)= b_3 i \Id_E,
\end{split}\]
where $b_1, b_2, b_3\in \R$. Then using the same argument in \eqref{eq:YM-min}, we have
\[\begin{split}
\YM(\nabla')&=-\tau_E(\{\Theta_{\nabla'}, \Theta_{\nabla'}\})=-\tau_E((\Theta_{\nabla'}(X,Y))^2+(\Theta_{\nabla'}(X,Z))^2+(\Theta_{\nabla'}(Y,Z))^2)\\
&=\frac{2\pi^2}{\mu}-\tau_E(-b_1^2 \Id_E-b_2^2 \Id_E -b_3^2 \Id_E)\\
&=\frac{2\pi^2}{\mu}+2\mu(b_1^2+b_2^2+b_3^2) \; \ge \frac{2\pi^2}{\mu} = \YM(\nabla).
\end{split}\]
Therefore, $\nabla$ is a minimizer of $\YM$ subject to the constant curvature constraint, which completes the proof.

\end{proof}

\begin{rmk}\label{rmk:const-conn-not-min}
\begin{itemize}
\item[(a)] Observe that if $\nabla$ is a compatible connection with constant curvature of the form given in \eqref{eq:curv-min} and $\nabla$ is a minimizer of $\YM$ subject to the constant curvature constraint, in the sense of Theorem~\ref{prop:min-cp}, then $\nabla$ is a critical point of $\YM$.
\item[(b)] We can also use Theorem~\ref{prop:min-cp} to see that the compatible connection $\nabla^1$ with constant curvature given in Theorem~\ref{thm:const_conn2} does not give a minimizer of $\YM$ since $\Theta_{\nabla^1}(X,Y)=\nu i \Id_E\ne 0$.
\end{itemize}
\end{rmk}

By combining Theorem~\ref{thm:critical-ioi} and Theorem~\ref{prop:min-cp} together, we characterize Yang-Mills connections with constant curvature as follows. 

\begin{cor}\label{thm:ym-connection}
Let $\Xi$ be the left $\EE$ and right $\DD$ projective bimodule described in Section~\ref{sec:prelim} .
Then $\nabla$ is a Yang-Mills connection on $\Xi$ with constant curvature $\Theta_\nabla$  if and only if  $\nabla$ is a compatible connection on $\Xi$ over $\DD$ with the following form of constant curvature $\Theta_\nabla$:
\begin{equation}\label{eq:thm:const-curv}
\Theta_\nabla(X,Y)=0,\quad \Theta_\nabla(X,Z)=0, \quad \Theta_\nabla(Y,Z)=\frac{\pi i}{\mu} \Id_E.
\end{equation}
\end{cor}

\begin{proof}
Suppose that $\nabla$ is a Yang-Mills connection on $\Xi$ with constant curvature. This means that $\nabla$ is a critical point and a minimizer of $\YM$. Then Theorem~\ref{prop:min-cp} implies that the curvature of $\nabla$ has the form
\[
\Theta_\nabla(X,Y)=0,\quad \Theta_\nabla(X,Z)=0, \quad \Theta_\nabla(Y,Z)=\frac{\pi i}{\mu} \Id_E.
\]
So we are done.

 On the other hand, suppose that $\nabla$ is a compatible connection on $\Xi$ over $\DD$ with constant curvature as in \eqref{eq:thm:const-curv}. Then Theorem~\ref{prop:min-cp} implies that $\nabla$ is a minimizer of $\YM$ subject to the constant curvature constraint and Theorem~\ref{thm:critical-ioi} implies that $\nabla$ is a critical point of $\YM$, and hence $\nabla$ is a Yang-Mills connection with constant curvature.

\end{proof}

Now we investigate a class of skew-symmetric elements $\HH\in \EE$ that preserve the properties of critical points and the minimizing conditions of $\YM$ for the $\DD$ as follows.

\begin{prop}\label{prop:skew-H}
Let ${\Xi}$ be the left $\EE$ -- right $\DD$ projective bimodule with the right inner product $\langle \cdot, \cdot \rangle_R^D$ described in Section~\ref{sec:prelim} and let $\mathfrak{h}$ be the Heisenberg Lie algebra with basis $\{X,Y,Z\}$ with $[X,Y]=cZ$ given in \eqref{eq:basis-XYZ}. Let $\nabla$ be a compatible connection on $\Xi$ with respect to $\langle \cdot, \cdot \rangle_R^D$. Suppose that $\mathbb{H}$ is a linear map from $\mathfrak{h}$ into  $\EE$, and suppose that for each $X\in \mathfrak{h}$, each element $\mathbb{H}_X$ of $\EE$ has the form $\mathbb{H}_X(x,y,q)=i \,T(x,y)\delta_0(q)$, where $T$ is a real-valued differentiable function on $\R\times \T$ with $T(x-2p\mu, y-2p\nu)=T(x,y)$ for $p\in \Z$. Then $\mathbb{H}_X$ is skew-symmetric in the sense that $\mathbb{H}_X^\ast=-\mathbb{H}_X$ for each $X\in \mathfrak{h}$, and $\nabla+\mathbb{H}$ is a compatible connection on $\Xi$ with respect to $\langle \cdot, \cdot \rangle_R^D$.
\end{prop}

\begin{proof}
A straightforward computation shows that $\mathbb{H}_X^\ast=-\mathbb{H}_X$ for each $X\in \mathfrak{h}$, so we leave it to the readers. Since we assumed that $\nabla$ is a compatible connection and every other compatible connection should have a form $\nabla+\mathbb{F}$, where $\mathbb{F}$ is a skew-symmetric element of $\EE$, $\nabla+\mathbb{H}$ is a compatible connection as discussed in \cite{CR}.\footnote{One can check that $\nabla+\mathbb{H}$ satisfies \eqref{eq:nabla-der} and \eqref{eq:nabla-comp} directly using the explicit formulas of the action and inner product $\langle \cdot, \cdot \rangle_R^D$.} 

\end{proof}

\begin{rmk}
One might wonder if the result of Proposition~\ref{prop:skew-H} holds for more general skew-symmetric elements $\HH \in \EE$. In fact, any element $\HH$ of $\EE$ is of the form $\HH(x,y,p)=\sum_{i\in \Z} H_i(x,y)\delta_i(p)$. So one might think getting concrete functions $H_i$ that satisfies  $\HH^*=-\HH$ in addition to the fixed point condition $\gamma_k(\HH)=\HH$ for all $k\in \Z$ given in \eqref{eq:fixed-cond} might not look so difficult.  However, even for the second simplest case, $\mathbb{H}$ is supported on $-1$ and $1$, it seems to be highly nontrivial to obtain concrete conditions on the functions $H_i$ and to prove existence of those functions.
\end{rmk}

Using the above proposition, we construct a new class of Yang-Mills connections with constant curvature on $\Xi$ over $\DD$ using $\nabla^0$ given in \eqref{eq:min_conn}.

\begin{thm}\label{thm:YM-connection}
Let $\Xi$ be the left $\EE$ and right $\DD$ projective bimodule described in Section~\ref{sec:prelim} and let $\mathfrak{h}$ be the Heisenberg Lie algebra with basis $\{X,Y,Z\}$ with $[X,Y]=cZ$ given in \eqref{eq:basis-XYZ}. Suppose that $\mu\ne 0$ and $\nu\ne 0$. Let $\nabla^0$ be the compatible connection given in \eqref{eq:min_conn} with constant curvature given in \eqref{eq:min_curv}. Let $\mathbb{H}$ be the linear map from $\mathfrak{h}$ to the set of skew-symmetric elements of $\EE$ given by
\[\begin{split}
&\mathbb{H}_X (x,y,p)= i\, g_1(y) \delta_0(p)\\
&\mathbb{H}_Y(x,y,p)= i\, g_2(x) \delta_0(p)\\
&\mathbb{H}_Z(x,y,p)=0,
\end{split}\] 
where $(x,y)\in \RR\times \TT$, and $g_1$, $g_2$ are real-valued differentiable functions satisfying $g_1(y)=g_1(y-2p\nu)$, $g_2(x)=g_2(x-2p\mu)$.
\begin{itemize}
\item[(a)] The connection $\nabla=\nabla^0+ \mathbb{H}$ is a compatible with respect to $\langle \cdot, \cdot \rangle_R^D$, and has constant curvature given by
\[
\Theta_{\nabla}(X,Y)=0,\quad \Theta_{\nabla}(X,Z)=0,\quad \Theta_{\nabla}(Y,Z)=\frac{\pi i}{\mu}\Id_E,
\]
where $\Id_E(x,y,p)=\delta_0(p)$.
\item[(b)] The connection $\nabla=\nabla^0+\mathbb{H}$ is a Yang-Mills connection with constant curvature on $\Xi$ over $\DD$.
\end{itemize}
\end{thm}

\begin{proof}
For (a), first notice that $\nabla=\nabla^0+\HH$ is a compatible connection, since $\mathbb{H}$ satisfies conditions in Proposition~\ref{prop:skew-H}. To compute the corresponding curvature $\Theta_\nabla$, notice that
\[\begin{split}
&\Theta_{\nabla}(X,Y)=\Theta_{\nabla^0}(X,Y)+[\nabla^0_X, \HH_Y]-[\nabla^0_Y, \HH_X]+[\HH_X, \HH_Y]-\HH_{[X,Y]}\\
&\Theta_{\nabla}(X,Z)=\Theta_{\nabla^0}(X,Z)+[\nabla^0_X, \HH_Z]-[\nabla^0_Z, \HH_X]+[\HH_X,\HH_Z]\\
&\Theta_\nabla(Y,Z)=\Theta_{\nabla^0}(Y,Z)+[\nabla^0_Y, \HH_Z]-[\nabla^0_Z, \HH_Y]+[\HH_Y, \HH_Z]
\end{split}\]
Then since $\HH_X$ and $\HH_Y$ are skew-symmetric multiplication-type elements of $\EE$, $\HH_X$ is given by a  function of $y$, and $\HH_Y$ is given by a function of $x$, Lemma~\ref{lem:nabla_G} implies that 
\[
[\nabla^0_X, \HH_Y]=0, \quad [\nabla^0_Y, \HH_X]=0, \quad [\nabla^0_Z, \HH_X]=0,\quad [\nabla^0_Z, \HH_Y]=0.
\]
Also we have $[\HH_X, \HH_Y]=0$, $[\HH_X,\HH_Z]=0$ and $[\HH_Y, \HH_Z]=0$  since $\HH_X$ and $\HH_Y$ are multiplication-type elements of $\EE$, and $\HH_Z=0$.
Also since $\Theta_{\nabla^0}(X,Y)=0$, $\Theta_{\nabla^0}(X,Z)=0$ and $\Theta_{\nabla^0}(Y,Z)=\frac{\pi i}{\mu}\Id_E$, where $\Id_E$ is the identity element of $\EE$, we have
\[
\begin{split}
&\Theta_{\nabla}(X,Y)=\Theta_{\nabla^0}(X,Y)=0\\
&\Theta_{\nabla}(X,Z)=\Theta_{\nabla^0}(X,Z)=0\\
&\Theta_\nabla(Y,Z)=\Theta_{\nabla^0}(Y,Z)=\frac{\pi i}{\mu}\Id_E,\\
\end{split}\]
which proves (a).

Then  (a) implies that the connection $\nabla=\nabla^0+ \HH$ is a Yang-Mills connection on $\Xi$ over $\DD$ by Theorem~\ref{thm:ym-connection}, which proves (b). 

\end{proof}

We next give concrete examples of Yang-Mills connections on $\DD$ constructed using the preceding result.

\begin{example}\label{ex:YM-conn}
According to Theorem~\ref{thm:YM-connection}, there are many Yang-Mills connections with constant curvature on $\Xi$ over $\DD$. For example, for $\mu\ne 0$ and $\nu\ne 0$, let $\HH_X(x,y,q)=i\, g_1(y)\delta_0(q)$ and $\HH_Y(x,y,q)=i\, g_2(x) \delta_0(q)$ with

\[\begin{split}
&g_1(y)=\cos^{n_1}(2\pi y) \sin^{m_1} (2 \pi y)\\
&g_2(x)=\cos^{n_2}(\frac{\alpha_1 \pi x}{\mu}) \sin^{m_2} (\frac{\alpha_2 \pi x}{\mu}),
\end{split}\]

where $n_1, m_1, n_2, m_2$ are non-negative integers (but not all pairs $(n_1, m_1)$ and $(n_2, m_2)$ are trivial) and $\alpha_1, \alpha_2$ are nonzero real numbers. Then $g_1$ satisfies $g_1(y)=g_1(y-2p\nu)$ and $g_1(y)=g_1(y-1)$, and $g_2$ satisfies $g_2(x)=g_2(x-2p\mu)$.  Thus $\nabla^0+\HH$ is a Yang-Mills connection on $\Xi$ over $\DD$ for each $n_1, n_2, m_1, m_2, \alpha_1, \alpha_2$.

In particular, a connection $\nabla$ given by
\[\begin{split}
(\nabla_Xf)(x,y)&=-\frac{\partial f}{\partial y}(x,y)+\frac{\pi c i}{2\mu}x^2 f(x,y) +  i \cos (2 \pi y) f(x,y)\\
(\nabla_Yf)(x,y)&=-\frac{\partial f}{\partial x}(x,y) + i \cos (\frac{\pi x}{\mu}) f(x,y)\\
(\nabla_Zf)(x,y)&=\frac{\pi i x}{\mu}f(x,y)
\end{split}\]
is a Yang-Mills connection with constant curvature on $\Xi$ over $\DD$.
\end{example}

\section{Connections with non-constant curvature}\label{sec:non-const}

As discussed in Remark~\ref{rmk:const-conn-not-min}, if $\nabla$ is a minimizer of $\YM$ subject to the constant curvature constraint, then $\nabla$ is a critical point of $\YM$. However, this is no longer true for a compatible connection with non-constant curvature as shown in the following Theorem.

\begin{thm}\label{thm:min-but-nc}
Let $\Xi$ be the left $\EE$ and right $\DD$ projective bimodule described in Section~\ref{sec:prelim}, and let $\mathfrak{h}$ be the Heisenberg Lie algebra with basis $\{X,Y,Z\}$ with $[X,Y]=cZ$ given in \eqref{eq:basis-XYZ}, where $c$ is a positive integer. Suppose that $\mu\ne 0$. Let $\nabla^0$ be the compatible connection given in \eqref{eq:min_conn} with constant curvature given in \eqref{eq:min_curv}. Let $\mathbb{H}$ be the linear map from $\mathfrak{h}$ to the set of skew-symmetric elements of $\EE$ given by
\[\begin{split}
&\mathbb{H}_X (x,y,p)= 0\\
&\mathbb{H}_Y(x,y,p)= 0\\
&\mathbb{H}_Z(x,y,p)=i \cos(\frac{\alpha \pi x}{\mu})  \, \delta_0(p),
\end{split}\] 
where $\alpha \in \R\setminus \negthickspace \{0\}$. Then 
\begin{itemize}
\item[(a)] The connection $\nabla=\nabla^0+\mathbb{H}$ is compatible with respect to $\langle \cdot, \cdot \rangle_R^D$, and it has non-constant curvature such that $\Theta_\nabla(X,Y)\ne 0$.
\item[(b)] There exists a triple $(c, \mu,\alpha)$ with  $c\in \Z^+, \mu\in (0,1/2], \alpha\in \R\setminus \negthickspace \{0\}$ such that $\nabla$ is not a critical point but its value of $\YM$ satisfies $\YM(\nabla)  < \YM(\nabla^0)=\frac{2\pi^2}{\mu}$.
\end{itemize}
\end{thm}

\begin{proof}
Since $\nabla^0$ is a compatible connection with respect to $\langle \cdot, \cdot \rangle_R^D$ and $\mathbb{H}$ is skew-symmetric, $\nabla=\nabla^0+\mathbb{H}$ is a compatible connection with respect to $\langle \cdot, \cdot \rangle_R^D$. Using Lemma~\ref{lem:nabla_G}, we compute the curvature of $\nabla$ as follows.
\[\begin{split}
&\Theta_\nabla(X,Y)=-c\,\mathbb{H}_Z \ne 0, \\
&\Theta_\nabla(X,Z)=[\nabla^0_X, \mathbb{H}_Z]=0, \\
&\Theta_\nabla(Y,Z)=\frac{\pi i}{\mu}\Id_E+[\nabla^0_Y, \mathbb{H}_Z]=\frac{\pi i }{\mu} \Id_E+\frac{\partial \mathbb{H}_Z}{\partial x} \ne 0,
\end{split}\]
 which shows that $\Theta_\nabla$ cannot be constant because $\mathbb{H}_Z(x,y,p)=i \cos(\frac{\alpha \pi x}{\mu})  \, \delta_0(p)$ is not constant for $\alpha \in \R \setminus \negthickspace \{0\}$. This proves (a). 

For (b), first recall that 
\[
\YM(\nabla^0)=-\tau_E(\{\Theta_{\nabla^0}, \Theta_{\nabla^0}\})=-\tau_E((\frac{\pi i}{\mu} \Id_E)^2)=-\int_0^{2\mu}\int_0^1 (-\frac{\pi^2}{\mu^2})dy\, dx=\frac{2\pi^2}{\mu}.
\]

Since $\Theta_\nabla(X,Y)\ne 0$, Proposition~\ref{prop:critical} implies that $\nabla$ is not a critical point of $\YM$.
For the second assertion of (b), we compute
\[\begin{split}
\YM(\nabla)&=-\tau_E(\{\Theta_\nabla, \Theta_\nabla\})=-\tau_E( (c\, \mathbb{H}_Z)^2+(\frac{\pi i}{\mu}\Id_E+\frac{\partial \mathbb{H}_Z}{\partial x})^2)\\
&=-\tau_E( (c\, \mathbb{H}_Z)^2)-\tau_E(-\frac{\pi^2}{\mu^2}\Id_E+\frac{2\pi i}{\mu}\frac{\partial \mathbb{H}_Z}{\partial x}+ (\frac{\partial \mathbb{H}_Z}{\partial x})^2)\\
&=-\int_0^{2\mu}\int_0^1 c^2 (\mathbb{H}_Z\ast \mathbb{H}_Z)(x,y,0) \, dy \,dx-\int_0^{2\mu}\int_0^1 (-\frac{\pi^2}{\mu^2})\Id_E(x,y,0)\, dy\, dx \\
&\quad -\int_0^{2\mu}\int_0^1 \frac{2\pi i}{\mu} \, \frac{\partial \mathbb{H}_Z}{\partial x}(x,y,0)\, dy\, dx
-\int_0^{2\mu}\int_0^1(\frac{\partial \mathbb{H}_Z}{\partial x}\ast \frac{\partial \mathbb{H}_Z}{\partial x})(x,y,0)\, dy\, dx,
\end{split}\]
where $\ast$ is the convolution product of $\EE$.
We compute the above four terms separately as follows:
For the first term, notice that 
\[
(\mathbb{H}_Z\ast \mathbb{H}_Z)(x,y,0)=\sum_{q\in \Z}\mathbb{H}_Z(x,y,q) \mathbb{H}_Z(x+q, y, -q)=\mathbb{H}_Z(x,y,0)\mathbb{H}_Z(x,y,0)=-\cos^2(\frac{\alpha \pi x}{\mu}).
\]
So the first integral becomes 
\[
c^2\,\int_0^{2\mu}\int_0^1\cos^2(\frac{\alpha \pi x}{\mu})\, dy\, dx = c^2\mu + \frac{c^2\mu}{4\alpha \pi}\sin (4\alpha \pi)
\]
The second integral becomes
\[
\int_0^{2\mu}\int_0^1 \frac{\pi^2}{\mu^2} \, dy\, dx=\frac{2\pi^2}{\mu}.
\]
For the third integral, note that $\frac{\partial \mathbb{H}_Z}{\partial x}(x,y,0)=-\frac{\alpha \pi}{\mu} i\sin(\frac{\alpha \pi x}{\mu}) $. So the third integral becomes
\[
-\int_0^{2\mu}\int_0^1 \frac{2\pi i}{\mu} (-\frac{\alpha \pi}{\mu} i \sin(\frac{\alpha \pi x}{\mu}) )\, dy\, dx
=-\frac{2\pi^2\alpha}{\mu^2}\int_0^{2\mu} \sin (\frac{\alpha \pi x}{\mu})\, dx = \frac{2\pi}{\mu} (\cos (2\alpha \pi)-1)
\]
For the fourth integral, note that 
\[\begin{split}
(\frac{\partial \mathbb{H}_Z}{\partial x}\ast \frac{\partial \mathbb{H}_Z}{\partial x})(x,y,0)&=\sum_{q\in \Z} \frac{\partial \mathbb{H}_Z}{\partial x}(x,y,q) \frac{\partial \mathbb{H}_Z}{\partial x}(x+q, y,-q)=\frac{\partial \mathbb{H}_Z}{\partial x}(x,y,0) \frac{\partial \mathbb{H}_Z}{\partial x}(x,y,0) \\
&= \big(-\frac{\alpha \pi}{\mu} i\, \sin(\frac{\alpha \pi x}{\mu})\big)^2=-\frac{\alpha^2\pi^2}{\mu^2} \sin^2 (\frac{\alpha \pi x}{\mu}).
\end{split}\]
Thus the fourth integral becomes

\[\begin{split}
&-\int_0^{2\mu}\int_0^1 \big(-\frac{\alpha^2\pi^2}{\mu^2} \sin^2 (\frac{\alpha \pi x}{\mu})\big)\, dy\, dx = \frac{\alpha^2\pi^2}{\mu^2}\int_0^{2\mu} \big( \frac{1}{2}-\frac{1}{2} \cos(\frac{2\alpha \pi x}{\mu}) \big)\, dx\\
&\quad =\frac{\alpha^2\pi^2}{2\mu^2}\big(2\mu - \frac{\mu}{2\alpha \pi}\sin (4\alpha \pi)\big)=\frac{\alpha^2\pi^2}{\mu}-\frac{\alpha \pi}{4\mu}\sin(4 \alpha \pi).
\end{split}\]

Hence we have
\begin{equation}\label{eq:YM-value}
\YM(\nabla)=c^2\mu +\frac{c^2\mu}{4\alpha \pi}\sin(4\alpha \pi)+\frac{2\pi^2}{\mu}+\frac{2\pi}{\mu} (\cos (2\alpha \pi)-1)+\frac{\alpha^2\pi^2}{\mu}-\frac{\alpha \pi}{4\mu}\sin(4 \alpha \pi).
\end{equation}
 For simplicity, first choose $c=1$ and $\mu=\frac{1}{2}$. Then choose $\alpha=\frac{1}{8}$ so that so that $\sin(4\alpha\pi)=1$. Then we have
\[\begin{split}
\YM(\nabla)&=\frac{1}{2}+\frac{1}{\pi}+4\pi^2+ 4\pi (\frac{\sqrt{2}}{2}-1)+\frac{\pi^2}{32}-\frac{\pi}{16}<4\pi^2=\frac{2\pi^2}{\mu}=\YM(\nabla^0),
\end{split}\]
which gives the desired result.
\end{proof}

\begin{rmk}
Observe that for given connection $\nabla$ in (a) of Theorem~\ref{thm:min-but-nc}, the value of the Yang-Mills functional $\YM(\nabla)$ depends on $c\in \Z^+$ and $\mu\in (0,\frac{1}{2}]$. This means only for certain $\DD$, this particular form of connection $\nabla$ gives a smaller value of $\YM(\nabla)$ than that of $\YM(\nabla^0)$.
\end{rmk}

\section{A projective module with trace $2\nu$}\label{sec:trace_2nu}

It is well-known that the left $\EE$--right $\DD$ projective module $\Xi$ has trace $2\mu$ if $\mu>0$. In particular, one can compute that the projection $Q=\langle R, R\rangle_R^D$ gives 
\[
\tau(Q)=\tau(\langle R, R\rangle_R^D)=\tau_E(\langle R, R \rangle_L^E)=\tau_E(\Id_E)=2\mu
\]
if $\mu>0$. One might wonder if we could construct a projective module over $\DD$ with trace other than $2 \mu$ (meaning possibly in a different $K_0$-class) and if we could find a meaningful compatible connection on it. In this section, we investigate compatible connections on a projective module with trace $2\nu$.

\subsection{The Grassmannian connection}

Theorem~1 of \cite{Ab5} implies that if $0\le \nu \le 1/2$ and $\mu>1$, then there exists a finitely generated right $\DD$ projective submodule $\Xi'$ of $\Xi$ with trace $2\nu$. The theorem does not give an explicit construction of the module but it gives an idea of how one can construct the corresponding projection $P$, and we find the Grassmannian connection as follows.

\begin{prop}\label{prop:Grass_conn}
Let $\Xi$ be the left-$\EE$ and right $\DD$ projective bimodule given in Section~\ref{sec:prelim} and let $\mathfrak{h}$ be the Heisenberg Lie algebra with basis $\{X,Y,Z\}$ with $[X,Y]=cZ$ given in \eqref{eq:basis-XYZ}.  Let $R\in \Xi$ be such that $\langle R, R\rangle_R^D=Q$ and $\langle R, R\rangle_L^E=Id_E$. Let $P$ be the projection in $\EE$ with trace $2\nu$ constructed by Abadie in \cite{Ab5}, and thus $P\Xi$ is a submodule of $\Xi$.
Then
\begin{itemize}
\item[(a)] The Grassmannian connection on $P\Xi$ is given by
\[
\widehat{\nabla}_X(\xi)=(P\cdot R)\delta_X(\langle R, P\cdot \xi\rangle_R^D)
\]
for $X\in \mathfrak{h}$.
\item[(b)] The Grassmannian connection $\widehat{\nabla}_X$ is compatible for $\delta$ with respect to $\langle \cdot,\cdot\rangle_R^D$. i.e. for $\xi,\eta\in P\Xi$, $\widehat{\nabla}_X$ satisfies
\[
\langle \widehat{\nabla}_X(\xi), \eta\rangle_R^D+\langle \xi,\widehat{\nabla}_X(\eta)\rangle_R^D=\delta_X(\langle \xi,\eta\rangle_R^D
\]
for $X \in \mathfrak{h}$.
\end{itemize}
\end{prop}
\begin{proof}
For (a), fix $\phi \in \DD$ and $\xi\in P\Xi$. We compute
\[\begin{split}
\widehat{\nabla}_X(\xi \cdot \Phi)&=(P\cdot R)\cdot \delta_X(\langle R, P\cdot (\xi\cdot \Phi)\rangle_R^D)\\
&=(P\cdot R)\cdot \delta_X(\langle R, P\cdot \xi\rangle_R^D\cdot \Phi)\\
&=(P\cdot R)\big(\delta_X(\langle R, P\cdot \xi \rangle_R^D\cdot \Phi+\langle R, P\cdot \xi\rangle_R^D\cdot \delta_X(\Phi)\big)\\
&=(P\cdot R)\delta_X(\langle R, P\cdot \xi\rangle_R^D)\cdot \Phi + (P\cdot R)\langle R, P\cdot \xi\rangle_R^D\cdot \delta_X(\Phi)\\
&=\widehat{\nabla}_X(\xi)\cdot \Phi + P\cdot \langle R, R\rangle_L^E\cdot (P\cdot \xi) \cdot \delta_X(\Phi)
\end{split}\]
Since $\langle R, R\rangle_L^E= 1$ and $P\cdot \xi=\xi$,
\[
\widehat{\nabla}_X(\xi \cdot \Phi)=\widehat{\nabla}_X(\xi)\cdot \Phi + \xi \cdot \delta_X(\Phi).
\]
Thus $\widehat{\nabla}$ is a connection.

For (b), fix $\xi,\eta\in P\Xi$. We proceed as follows:
\[\begin{split}
&\langle \widehat{\nabla}_X(\xi),\eta\rangle_R^D + \langle \xi, \widehat{\nabla}_X(\eta)\rangle_R^D\\
&=\langle (P\cdot R)\cdot \delta_X(\langle R, P\cdot \xi\rangle_R^D),\eta\rangle_R^D + \langle \xi, (P\cdot R)\cdot \delta_X(\langle R, P\cdot \eta\rangle_R^D)\rangle_R^D\\
&=\big(\langle \eta, (P\cdot R)\rangle_R^D\cdot \delta_X(\langle R, P\cdot \xi\rangle_R^D\big)^* + \langle \xi, P\cdot R\rangle_R^D\cdot \delta_X(\langle R, P\cdot \eta\rangle_R^D)\\
&=\delta_X(P\cdot \xi, \eta\rangle_R^D\cdot \langle P\cdot R, \eta\rangle_R^D + \langle \xi, P\cdot R\rangle_R^D\cdot \delta_X(\langle R, P\cdot \eta\rangle_R^D)\\
\end{split}\]
Since $\langle \xi, P\cdot R\rangle_R^D=\langle P\cdot \xi, R\rangle_R^D$ and $\langle R, P\cdot \eta\rangle_R^D=\langle P\cdot R, \eta\rangle_R^D$, the above equation becomes
\[\begin{split}
&=\delta_X(\langle P\cdot \xi, R\rangle_R^D)\cdot \langle P\cdot R, \eta\rangle_R^D + \langle P\cdot \xi, R\rangle_R^D\cdot \delta_X(\langle P\cdot R, \eta\rangle_R^D)\\
&=\delta_X(\langle P\cdot \xi, R\rangle_R^D\cdot \langle P\cdot R, \eta\rangle_R^D)=\delta_X(\langle P\cdot \xi, R\cdot \langle R, P\cdot \eta\rangle_R^D \rangle_R^D) \\
&=\delta_X(\langle P\cdot \xi, P\cdot \eta\rangle_R^D) =\delta_X(\langle \xi, \eta\rangle_R^D),
\end{split}\]
which proves the desired result.
\end{proof}
We determine the curvature of the Grassmannian connection $\widehat{\nabla}$ as follows.
\begin{prop}\label{prop:Grass-curvature}
Let $\widehat{\nabla}$ be the Grassmannian connection given in Proposition~\ref{prop:Grass_conn}. Let $S=P\cdot R$. Then the curvature ${\Theta}_{\widehat{\nabla}}$ of $\widehat{\nabla}$ is $P\EE P$-valued and is given by
\[
{\Theta}_{\widehat{\nabla}}(X,Y)=\langle S\cdot (\delta_X\langle S, S\rangle_R^D\cdot \delta_Y\langle S, S\rangle_R^D-\delta_Y\langle S, S\rangle_R^D\cdot \delta_X\langle S, S\rangle_R^D), S\rangle_L^E.
\]

\end{prop}

\begin{proof}
Fix $\zeta\in P\Xi$, then we compute
\[\begin{split}
&\widehat{\nabla}_X\widehat{\nabla}_Y(\zeta)=(P\cdot R)(\delta_X\langle R, P\cdot \widehat{\nabla}_Y(\zeta)\rangle_R^D)\\
&=(P\cdot R)(\delta_X\langle R, P\cdot(P\cdot R)\cdot \delta_Y\langle R, (P\cdot \zeta)\rangle_R^D\rangle_R^D)\\
&=(P\cdot R)(\delta_X(\langle R, P\cdot R\rangle_R^D\cdot \delta_Y\langle R, P\cdot \zeta\rangle_R^D))\\
&=(P\cdot R)(\delta_X(\langle R, P\cdot R\rangle_R^D)\cdot \delta_Y(\langle R, P\cdot \zeta\rangle_R^D) + \langle R, P\cdot R\rangle_R^D\cdot \delta_X\delta_Y\langle R, P\cdot \zeta\rangle_R^D)
\end{split}\]
Similarly, we have
\[\begin{split}
&\widehat{\nabla}_Y\widehat{\nabla}_X(\zeta)\\
&=(P\cdot R)(\delta_Y(\langle R, P\cdot R\rangle_R^D)\cdot \delta_X(\langle R, P\cdot \zeta\rangle_R^D) + \langle R, P\cdot R\rangle_R^D\cdot \delta_Y\delta_X\langle R, P\cdot \zeta\rangle_R^D)
\end{split}\]
Also
\[
\widehat{\nabla}_{[X,Y]}(\zeta)=(P\cdot R)\cdot \delta_{[X,Y]}\langle R, P\cdot \zeta\rangle_R^D.
\]
Thus
\[\begin{split}
&{\Theta}_{\widehat{\nabla}}(X,Y)\cdot \zeta=\widehat{\nabla}_X\widehat{\nabla}_Y(\zeta)-\widehat{\nabla}_Y\widehat{\nabla}_X(\zeta)-\widehat{\nabla}_{[X,Y]}(\zeta)\\
&=(P\cdot R)(\delta_X\langle R, P\cdot R\rangle_R^D\cdot \delta_Y\langle R, P\cdot \zeta\rangle_R^D-\delta_Y\langle R, P\cdot R\rangle_R^D\cdot \delta_X\langle R, P\cdot \zeta\rangle_R^D).
\end{split}\]

Now we claim that for $\xi\in P\Xi$, we have
\[\begin{split}
&(P\cdot R)\big(\delta_X\langle R, P\cdot R\rangle_R^D\cdot \delta_Y\langle R, P\cdot \xi\rangle_R^D-\delta_Y\langle R, P\cdot R\rangle_R^D\cdot \delta_X\langle R, P\cdot \xi\rangle_R^D\big)\\
&=\langle P\cdot R\cdot (\delta_X\langle R, P\cdot R\rangle_R^D\cdot \delta_Y\langle R, P\cdot R\rangle_R^D-\delta_Y\langle P\cdot R, R\rangle_R^D\cdot \delta_X\langle P\cdot R, R\rangle_R^D), P\cdot R\rangle_L^E\cdot \xi.
\end{split}\]
To see this, note that the value of ${\Theta}_{\widehat{\nabla}}(X,Y)$ lies in $(\EE)^s$, the set of skew-symmetric elements of $\EE$. So ${\Theta}_{\widehat{\nabla}}(X,Y)$ satisfies
\[
\langle {\Theta}_{\widehat{\nabla}}(X,Y)\cdot \xi, \eta\rangle_R^D+\langle \xi, {\Theta}_{\widehat{\nabla}}(X,Y)\cdot \eta\rangle_R^D=0\quad\text{for all $\xi,\eta\in P\Xi$}.
\]
We compute
\[\begin{split}
&\langle {\Theta}_{\widehat{\nabla}}(X,Y)\cdot \xi, \eta\rangle_R^D\\
&=\langle P\cdot R(\delta_X\langle R, P\cdot R\rangle_R^D\cdot \delta_Y\langle R, P\cdot \xi\rangle_R^D-\delta_Y\langle R, P\cdot R\rangle_R^D\cdot \delta_X\langle R, P\cdot \xi\rangle_R^D), \eta\rangle_R^D\\
&=(\langle \eta, P\cdot R\rangle_R^D(\delta_X\langle R, P\cdot R\rangle_R^D\cdot \delta_Y\langle R, P\cdot \xi\rangle_R^D-\delta_Y\langle R, P\cdot R\rangle_R^D\cdot \delta_X\langle R, P\cdot \xi\rangle_R^D))^*\\
&=(\delta_Y\langle P\cdot \xi, R\rangle_R^D\cdot \delta_X\langle P\cdot R, R\rangle_R^D-\delta_X\langle P\cdot R, R\rangle_R^D\cdot \delta_Y\langle P\cdot R, P\cdot R\rangle_R^D)\cdot \langle P\cdot R, \eta\rangle_R^D.
\end{split}\]
Similarly,
\[\begin{split}
&\langle \xi, {\Theta}_{\widehat{\nabla}}(X,Y)\cdot \eta\rangle_R^D\\
&=\langle \xi, P\cdot R\rangle_R^D(\delta_X\langle R, P\cdot R\rangle_R^D\cdot \delta_Y\langle R, P\cdot \eta\rangle_R^D-\delta_Y\langle R, P\cdot R\rangle_R^D\cdot \delta_X\langle R, P\cdot \eta\rangle_R^D)
\end{split}\]
Thus
\[\begin{split}
&\langle \xi, P\cdot R\rangle_R^D(\delta_X\langle R, P\cdot R\rangle_R^D\cdot \delta_Y\langle R, P\cdot \eta\rangle_R^D-\delta_Y\langle R, P\cdot R\rangle_R^D\cdot \delta_X\langle R, P\cdot \eta\rangle_R^D)\\
&=(\delta_X\langle P\cdot \xi, R\rangle_R^D\cdot \delta_Y\langle P\cdot R, R\rangle_R^D - \delta_Y\langle P\cdot \xi, R\rangle_R^D\cdot \delta_X\langle P\cdot R, R\rangle_R^D)\cdot \langle P\cdot R, \eta\rangle_R^D
\end{split}\]
Since $\xi\in \Xi$ is arbitrary, we choose $\xi= P\cdot R$ and apply $P\cdot R$ on the left of the above equation. Then we obtain
\[\begin{split}
&(P\cdot R)\cdot (\delta_X\langle R, P\cdot R\rangle_R^D\cdot \delta_Y\langle R, P\cdot \eta\rangle_R^D-\delta_Y\langle R, P\cdot R\rangle_R^D\cdot \delta_X\langle R, P\cdot \eta\rangle_R^D)\\
&= \langle (P\cdot R)\cdot (\delta_X\langle P\cdot R, R\rangle_R^D\cdot \delta_Y\langle P\cdot R, R\rangle_R^D-\delta_Y\langle P\cdot R, R\rangle_R^D\cdot \delta_X\langle P\cdot R, R\rangle_R^D), P\cdot R\rangle_L^E\cdot \eta,
\end{split}\]
which proves the claim. Therefore
\[\begin{split}
&{\Theta}_{\widehat{\nabla}}(X,Y)\cdot \zeta\\
&=(P\cdot R)(\delta_X\langle R, P\cdot R\rangle_R^D\cdot \delta_Y\langle R, P\cdot \zeta\rangle_R^D-\delta_Y\langle R, P\cdot R\rangle_R^D\cdot \delta_X\langle R, P\cdot \zeta\rangle_R^D)\\
&=\langle (P\cdot R)\cdot (\delta_X\langle P\cdot R, R\rangle_R^D\cdot \delta_Y\langle P\cdot R, R\rangle_R^D-\delta_Y\langle P\cdot R, R\rangle_R^D\cdot \delta_X\langle P\cdot R, R\rangle_R^D), P\cdot R\rangle_L^E\cdot \zeta.
\end{split}\]
Thus
\[\begin{split}
&{\Theta}_{\widehat{\nabla}}(X,Y)\\
&=\langle (P\cdot R)\cdot (\delta_X\langle P\cdot R, R\rangle_R^D\cdot \delta_Y\langle P\cdot R, R\rangle_R^D-\delta_Y\langle P\cdot R, R\rangle_R^D\cdot \delta_X\langle P\cdot R, R\rangle_R^D), P\cdot R\rangle_L^E.
\end{split}\]
Consequently, the value of ${\Theta}_{\widehat{\nabla}}(X,Y)$ lies in $P\EE P$.

Since $P$ is a projection, $\langle P\cdot R, R\rangle_R^D=\langle P\cdot R, P\cdot R\rangle_R^D$. Thus by letting $S=P\cdot R$ we obtain
\[\begin{split}
&{\Theta}_{\widehat{\nabla}}(X,Y)=\langle S\cdot (\delta_X\langle S, S\rangle_R^D\cdot \delta_Y\langle S, S\rangle_R^D-\delta_Y\langle S, S\rangle_R^D\cdot \delta_X\langle S, S\rangle_R^D), S\rangle_L^E,
\end{split}\]
which proves the desired result.
\end{proof}

%\begin{rmk}
For a given $S=P\cdot R$, computing the curvature $\Theta_{\widehat{\nabla}}$ explicitly is very complicated. However it is possible to verify that the curvature ${\Theta}_{\widehat{\nabla}}(x,y,p)\ne 0$ for $p\ne 0$. Thus it differs from the Grassmannian curvature $\Theta_{\nabla^G}$ of $\nabla^G$ on $\Xi,$ since $\Theta_{\nabla^G}$ is only supported on $p=0$.
%\end{rmk}

\subsection{Coupling constants associated to $P\EE\otimes_{\EE}\Xi$}

Let $A$ and $B$ be unital $C^*$-algebras, and let $X$ be a left-$A$ and right-$B$ projective bimodule. Let $\tau$ be a normalized faithful trace on $A$ and let $\Ind_X(\tau)$ be the trace on $B$ induced by $X$. According to \cite{P1}, the coupling constant from $A$ to $B$ for $\tau$ determined by $X$, denoted by $C_A^B(X)(\tau)$, is given by
\[
C_A^B(X)(\tau)=\Ind_X(\tau)(\Id_B)\tau(\Id_A)^{-1}.
\]
Similarly, let $\tau'$ be a faithful trace on $B$, then we define the coupling constant from $B$ to $A$ for $\tau'$ determined by $X$ by
\[
C_B^A(X)(\tau')=\Ind_X(\tau')(\Id_A)\tau'(\Id_B)^{-1}.
\]
For example, let $\Xi$ be the left-$\EE$ and right-$\DD$ projective module given before. We know that $\Xi$ is isomorphic to the left-$Q\DD Q$ and right $\DD$ projective bimodule $Q\DD$, where $Q=\langle R, R\rangle_D$ is the projection in $\DD$ and the function $R\in \Xi$ satisfies $\langle R,R\rangle_E=\Id_E$. Let $\tau_D$ be the normalized faithful trace on $\DD$ given before. Then the induced trace $\tau_E$ on $\EE$ is given by $\tau_E(\langle f,g\rangle_E)=\tau_D(\langle g,f\rangle_D)$ for $f,g\in \Xi$, and $\tau_E(\Id_E)=\tau_D(Q)=2\mu$. So we have
\[
C_{\DD}^{\EE}(\Xi)(\tau_D)=\Ind_\Xi(\tau_D)(\Id_{\EE})\tau_D(\Id_{\DD})^{-1}=\tau_E(\Id_{\EE})\tau_D(\Id_{\DD})^{-1}=2\mu.
\]
Since the left $Q\DD Q$--right $\DD$ projective module $Q\DD$ is algebraically isomorphic to the left $\EE$--right $\DD$ projective module $\Xi$ via the map $F$ described in Lemma~\ref{lem:1}, we obtain
\[
C_{\DD}^{Q\DD Q}(Q\DD)(\tau_D)=\Ind_{Q\DD}(\tau_D)(\Id_{Q\DD Q})\tau_D(\Id_D)^{-1}=\tau_D(Q)=2\mu.
\]

Now let $\tau'_E$ be the faithful normalized trace on $\EE$. Then the discussion of Example~2.2 in \cite{P1} implies that
\begin{equation}\label{eq:CC_EDXi}
C^{\EE}_{\DD}(\Xi)(\tau'_E)=\frac{1}{2\mu}.
\end{equation}
Since $P$ is a projection of $\EE$ with trace $2\nu$ given in \cite{Ab5} and $\tau'$ is the normalized trace on $\EE$, we have $\tau'_E(P)=\frac{2\nu}{2\mu}$. So we have
\[
C_{\EE}^{P\EE P}(P\EE)(\tau'_E)=\Ind_{P\EE}(\tau'_E)(\Id_{P\EE P})\tau'_E(\Id_E)^{-1}=\tau'_E(P)=\frac{2\nu}{2\mu}=\frac{\nu}{\mu}.
\]
Thus the discussion of Example~2.2 of \cite{P1} implies that
\[
C_{P\EE P}^{\EE}(P\EE)(\tau'_{P\EE P})=\frac{\mu}{\nu},
\]
where $\tau'_{P\EE P}$ is the faithful normalized trace on $P\EE P$.

By Proposition~2.4 of \cite{P1}, we obtain
\[\begin{split}
C_{P\EE P}^{\DD}(P\EE\otimes_{\EE}\Xi)(\tau'_{P\EE P})&=C_{\EE}^{\DD}(\Xi)(\Ind_{P\EE}(\tau'_{P\EE P})) C_{P\EE P}^{\EE}(P\EE)(\tau'_{P\EE P})\\
&=\frac{\nu}{\mu}\cdot \frac{1}{2\nu}\cdot\frac{\mu}{\nu}=\frac{1}{2\nu}.
\end{split}\]
Thus
\[
C_{\DD}^{P\EE P}(P\EE\otimes_{\EE}\Xi)(\tau_D)=2\nu.
\]

\subsection{Connections on $P\EE \otimes_{\EE} \Xi$}
Following Lemma~5.1 of \cite{CR}, we will construct a connection on the balanced tensor product of projective modules $P\EE\otimes_{\EE} \Xi$ in this section.

Recall that  $\delta$ is the Heisenberg group action on $\DD$ and $\Xi$ is the left $\EE$ -- right $\DD$ projective module in Section \ref{sec:prelim}. Let $\nabla$ be the compatible connection with constant curvature given in \eqref{eq:min_conn}, and let $P$ be the projection with trace $2 \nu$ of \cite{Ab5} described in Section~\ref{sec:trace_2nu}. Then there is a canonical left $P\EE P$ -- right $\EE$ projective module $P\EE$ with the inner products given by: 
\begin{equation}\label{eq:inner-prod-E}
\langle f,g \rangle_R^E=f^* g \;\quad\text{and} \;\quad \langle f,g\rangle_L^{PEP}=fg^* \qquad\text{for} \;\; f,g \in P\EE.
\end{equation}
Let $\widehat{\delta}$ be a covariant derivative on $\EE$ given by
\[
\widehat{\delta}_X(T)=[\nabla_X, T]\quad\text{for $T\in \EE$.}
\]
Then we can find a Grassmannian connection $\nabla^E$ on $P\EE$ given as follows.

\begin{prop}
Let $\mathfrak{h}$ be the Heisenberg Lie algebra with basis given in \eqref{eq:basis-XYZ}.
For $X\in \mathfrak{h}$, we define a map $\nabla^E_X$ on $P\EE$ by
\[
\nabla^E_X(f)=P \widehat{\delta}_X(f).
\]
Then $\nabla^E$ is a compatible connection for $\widehat{\delta}$ with respect to $\langle \cdot,\cdot\rangle_R^E$ given in \eqref{eq:inner-prod-E}.
\end{prop}
\begin{proof}
Fix $f\in P\EE$ and $\Phi \in \EE$. Note that the right action of $\EE$ on $P\EE$ is just a $C^*$-product. We compute
\[\begin{split}
\nabla^E_X(f\cdot\Phi)&=P\widehat{\delta}_X(f\Phi)=P(\widehat{\delta}_X(f)\Phi+f\widehat{\delta}_X(\Phi))\\
&=P\widehat{\delta}_X(f)\Phi+Pf\widehat{\delta}_X(\Phi) =\nabla^E_X(f)\cdot \Phi+f\cdot \widehat{\delta}_X(\Phi)
\end{split}\]
since $f\in P\EE$ implies $Pf=f$. Thus $\nabla^E$ is a connection on $\EE$.

To see the compatibility, fix $f,g\in P\EE$. We have that
\[\begin{split}
&\langle \nabla^E_X(f), g\rangle_R^E+\langle f,\nabla^E_X(g)\rangle_R^D=\langle P\widehat{\delta}_X(f), g\rangle_R^E+\langle f, P\widehat{\delta}_X(g)\rangle_R^E\\
&=(P\widehat{\delta}_X(f))^*g+f^*(P\widehat{\delta}_X(g))=\widehat{\delta}_X(f)^*P^*g+f^*P\widehat{\delta}_X(g)\\
&=\widehat{\delta}_X(f)^*g+f^*\widehat{\delta}_X(g)\quad\text{since $f,g\in P\EE$}\\
&=\widehat{\delta}_X(f^*g)=\widehat{\delta}_X(\langle f,g\rangle_R^E),
\end{split}\]
which proves the result.
\end{proof}

\begin{prop}
Let $\mathfrak{h}$ be the Heisenberg Lie algebra with basis given in \eqref{eq:basis-XYZ}.
The curvature  $\Theta_{\nabla^E}$ of $\nabla^E$ acts on $P\EE$ by
\[
\Theta_{\nabla^E}(X,Y)=\widehat{\delta}_X(P)\widehat{\delta}_Y(P)-\widehat{\delta}_Y(P)\widehat{\delta}_X(P)
\]
for $X,Y\in \mathfrak{h}$.
\end{prop}

\begin{proof}
Fix $f,g\in P\EE$ and fix $X,Y\in \mathfrak{g}$. First we compute

\[\begin{split}
\Theta_{\nabla^E}(X,Y)\cdot g &=\nabla^E_X\nabla^E_Y(g)-\nabla^E_Y\nabla^E_X(g)-\nabla^E_{[X,Y]}(g)\\
&=P \widehat{\delta}_X(P \widehat{\delta}_Y(g))-P\widehat{\delta}_Y(P\widehat{\delta}_X(g))-P\widehat{\delta}_{[X,Y]}(g)\\
&=P(\widehat{\delta}_X(P)\widehat{\delta}_Y(g)+P\widehat{\delta}_X\widehat{\delta}_Y(g))-P(\widehat{\delta}_Y(P)\widehat{\delta}_X(g)
+P\widehat{\delta}_Y\widehat{\delta}_X(g))-P\widehat{\delta}_{[X,Y]}(g)\\
&=P\widehat{\delta}_X(P)\widehat{\delta}_Y(g)-P\widehat{\delta}_Y(P)\widehat{\delta}_X(g).\\
\end{split}\]

Since $P=P^2$ and $g=Pg$ for $g\in P\EE$, we have
\begin{equation}\label{eq:Pg1}
\begin{split}
\widehat{\delta}_X(P)\widehat{\delta}_Y(g)=\widehat{\delta}_X(P^2)\widehat{\delta}_Y(g)&=(P\widehat{\delta}_X(P)+\widehat{\delta}_X(P)P)\widehat{\delta}_Y(g),\\
&=P\widehat{\delta}_X(P)\widehat{\delta}_Y(g)+\widehat{\delta}_X(P)P\widehat{\delta}_Y(g)
\end{split}\end{equation}
and
\begin{equation}\label{eq:Pg2}
\begin{split}
\widehat{\delta}_X(P)\widehat{\delta}_Y(g)=\widehat{\delta}_X(P)\widehat{\delta}_Y(Pg) &=\widehat{\delta}_X(P) (\widehat{\delta}_Y(P)g+ P\widehat{\delta}_Y(g))\\
&=\widehat{\delta}_X(P)\widehat{\delta}_Y(P)g+\widehat{\delta}_X(P)P\widehat{\delta}_Y(g).
\end{split}\end{equation}

Putting \eqref{eq:Pg1} and \eqref{eq:Pg2} together then gives
\[
P\widehat{\delta}_X(P)\widehat{\delta}_Y(g)=\widehat{\delta}_X(P)\widehat{\delta}(P)g.
\]
Therefore we have
\[\begin{split}
\Theta_{\nabla^E}(X,Y)\cdot g & = P\widehat{\delta}_X(P)\widehat{\delta}_Y(g)-P\widehat{\delta}_Y(P)\widehat{\delta}_X(g)\\
&=\widehat{\delta}_X(P)\widehat{\delta}_Y(P)g -\widehat{\delta}_Y(P)\widehat{\delta}_X(P) g,
\end{split}\]
which completes the proof.

\end{proof}

Now we can construct a tensor product connection on $P\EE\otimes_{\EE} \Xi:$ 

\begin{prop}\label{prop:tensor_conn}
Let $\nabla$ be a compatible connection for $\widehat{\delta}$ with respect to $\langle \cdot,\cdot\rangle_R^E$ and $\nabla'$ be a compatible connection for $\delta$ with respect to $\langle \cdot,\cdot \rangle_R^D$.
Then the tensor product map $\nabla\otimes_{\EE}\nabla':=\nabla\otimes_{\EE} I_\Xi+I_{P\EE}\otimes_{\EE} \nabla'$ on $P\EE\otimes_{\EE} \Xi$ is a compatible connection for $\delta$ with respect to the inner product given by
\[
\langle f\otimes \xi, g\otimes \eta\rangle^{\otimes}:=\langle \xi, \langle g,f\rangle_R^E\cdot \eta\rangle_R^D.
\]
Let $\Theta_{\nabla\otimes_{\EE}\nabla'}$ be the curvature of $\nabla\otimes_{\EE}\nabla'$. Then
\[
\Theta_{\nabla\otimes_{\EE}\nabla'}(X,Y)=\Theta_\nabla(X,Y)\otimes_{\EE} I_\Xi + I_{P\EE}\otimes_{\EE} \Theta_{\nabla'} (X,Y).
\]
\end{prop}
\begin{proof}
The result follows by Lemma~5.1 of \cite{CR}. 
\end{proof}

\begin{thm}\label{thm:tensor-prod}
Let $\nabla$ be a compatible connection for $\widehat{\delta}$ with respect to $\langle \cdot,\cdot\rangle_R^E$ and $\nabla'$ be a compatible connection for $\delta$ with respect to $\langle \cdot,\cdot \rangle_R^D$.
Suppose $\nabla\otimes_{\EE} \nabla'$ is a compatible connection on the balanced tensor prodcut $P\EE\otimes_{\EE} \Xi$ with respect to $\langle \cdot, \cdot \rangle^{\otimes}$ with the curvature $\Theta_{\nabla\otimes_{\EE} \nabla'}$ given in Proposition~\ref{prop:tensor_conn}.
Then the connection $\nabla\otimes_{\EE}\nabla'$ is a critical point  of $\YM$ in $CC(P\EE\otimes_{\EE} \Xi)$ if and only if $\nabla$ and $\nabla'$ are critical points of $\YM$ in $CC(P\EE)$ and $CC(\Xi)$, respectively.
\end{thm}

\begin{proof}
Suppose first that $\nabla\otimes_{\EE}\nabla'$ is a critical point of $\YM$. Then $\nabla\otimes_{\EE}\nabla'$ satisfies
\begin{equation}\label{eq:ten_critical_1}
[\nabla_Y\otimes_{\EE}\nabla'_Y, \Theta_{\nabla\otimes_{\EE} \nabla'}(X,Y)]+[\nabla_Z\otimes_{\EE} \nabla'_Z, \Theta_{\nabla\otimes_{\EE} \nabla'}(X,Z)]=0.
\end{equation}

\begin{equation}\label{eq:ten_critical_2}
[\nabla_X\otimes_{\EE}\nabla'_X, \Theta_{\nabla\otimes_{\EE} \nabla'}(Y,X)]+[\nabla_Z\otimes_{\EE}\nabla'_Z, \Theta_{\nabla\otimes_{\EE} \nabla'}(Y,Z)]=0.
\end{equation}

\begin{equation}\label{eq:ten_critical_3}
[\nabla_X\otimes_{\EE}\nabla'_X, \Theta_{\nabla\otimes_{\EE} \nabla'}(Z,X)]+[\nabla_Y\otimes_{\EE}\nabla'_Y, \Theta_{\nabla\otimes_{\EE} \nabla'}(Z,Y)]-c\Theta_{\nabla\otimes_{\EE} \nabla'}(X,Y)=0.
\end{equation}

Observe that
\[
[\nabla_Y\otimes_{\EE} I_{\Xi}, \Theta_{\nabla}(X,Y)\otimes_{\EE} I_\Xi]=[\nabla_Y, \Theta_{\nabla}(X,Y)]\otimes_{\EE} I_\Xi,
\]
and
\[\begin{split}
&[\nabla_Y\otimes_{\EE} I_\Xi, I_{P\EE}\otimes_{\EE} \Theta_{\nabla'}(X,Y)] \\ 
&=(\nabla_Y\otimes_{\EE} I_{\Xi})(I_{P\EE} \otimes_{\EE}\Theta_{\nabla'}(X,Y))- (I_{P\EE} \otimes_{\EE}\Theta_{\nabla'}(X,Y))(\nabla_Y\otimes_{\EE} I_{\Xi})   \\
&=\nabla_Y\otimes_{\EE} \Theta_{\nabla'}(X,Y)-\nabla_Y\otimes_{\EE}\Theta_{\nabla'}(X,Y)=0.\\
\end{split}\]

Similarly, we have
\[
[I_{P\EE}\otimes \nabla'_Y, \Theta_\nabla(X,Y)\otimes_{\EE} I_\Xi ]=0.
\]

So, the first bracket of \eqref{eq:ten_critical_1} can be computed as
\[\begin{split}
&[\nabla_Y\otimes_{\EE} {\nabla'}_Y, \Theta_{\nabla\otimes_{\EE}\nabla'}(X,Y)]\\
&=[\nabla_Y\otimes I_{\Xi}+I_{P\EE}\otimes \nabla'_Y, \Theta_\nabla(X,Y)\otimes_{\EE} I_\Xi + I_{P\EE}\otimes_{\EE} \Theta_{\nabla'} (X,Y)] \\
&=[\nabla_Y\otimes I_{\Xi}, \Theta_\nabla(X,Y)\otimes_{\EE} I_\Xi ] + [\nabla_Y\otimes I_{\Xi}, I_{P\EE}\otimes_{\EE} \Theta_{\nabla'} (X,Y)] \\
& \quad \quad \quad + [I_{P\EE}\otimes \nabla'_Y, \Theta_\nabla(X,Y)\otimes_{\EE} I_\Xi ] + [I_{P\EE}\otimes \nabla'_Y, I_{P\EE}\otimes_{\EE} \Theta_{\nabla'} (X,Y)] \\
\end{split}\]

\[\begin{split}
&= [\nabla_Y, \Theta_{\nabla}(X,Y)]\otimes_{\EE} I_\Xi + 0 + 0 + I_{P\EE} \otimes_{\EE} [\nabla'_Y, \Theta_{\nabla'}(X,Y)]
\end{split}\]
Similarly, the second bracket of \eqref{eq:ten_critical_1} can be computed as
\[\begin{split}
&[\nabla_Z\otimes_{\EE} \nabla'_Z, \Theta_{\nabla\otimes_{\EE} \nabla'}(X,Z)]\\
&=[\nabla_Z, \Theta_\nabla(X,Z)]\otimes_{\EE} I_{\Xi} + I_{P\EE} \otimes_{\EE} [\nabla'_Z, \Theta_{\nabla'}(X,Z)]
\end{split}\]

Therefore, the left hand side of \eqref{eq:ten_critical_1} is given by
\[\begin{split}
&[\nabla_Y\otimes_{\EE} {\nabla'}_Y, \Theta_{\nabla\otimes_{\EE}\nabla'}(X,Y)] +[\nabla_Z\otimes_{\EE} \nabla'_Z, \Theta_{\nabla\otimes_{\EE} \nabla'}(X,Z)]\\
&=([\nabla_Y, \Theta_{\nabla}(X,Y)]+[\nabla_Z, \Theta_{\nabla}(X,Z)])\otimes_{\EE} I_\Xi +I_{P\EE}\otimes_{\EE} ([\nabla'_Y, \Theta_{\nabla'}(X,Y)]+[\nabla'_Z, \Theta_{\nabla'}(X,Z)])
\end{split}\]
Hence equation \eqref{eq:ten_critical_1} holds if and only if the following two equations hold:
\[
[\nabla_Y, \Theta_{\nabla}(X,Y)]+[\nabla_Z, \Theta_{\nabla}(X,Z)]=0\quad\text{and}
\]
\[
[\nabla'_Y, \Theta_{\nabla'}(X,Y)]+[\nabla'_Z, \Theta_{\nabla'}(X,Z)]=0.
\]
Similarly, one can show that \eqref{eq:ten_critical_2} and \eqref{eq:ten_critical_3} are equivalent to the rest of the critical point conditions for $\nabla$ and $\nabla'$ respectively.

Therefore $\nabla\otimes_{\EE} \nabla'$ is a critical point of $\YM$ in $CC(P\EE\otimes_{\EE} \Xi)$ if and only if $\nabla$ and $\nabla'$ are critical points for $\YM$ in $CC(P\EE)$ and $CC(\Xi)$ respectively.

\end{proof}

Unfortunately we have not yet managed to determine conditions under which a balanced tensor product connection minimizes its Yang-Mills functional. Since this involves cross-terms between the two components of the balanced tensor product connection, it is not necessarily sufficient that both of these two component connections be Yang-Mills connections.

\begin{appendices}
\section{Curvature computations}\label{app:sec:cc}
 In this section, we verify that the connection $\nabla^0$ given in \eqref{eq:min_conn} is indeed a correct formula for a compatible connection with the constant curvature $\Theta_{\nabla^0}$ given in \eqref{eq:min_curv} in our setting.

\begin{prop}\label{App:prop:cc}
Let $\Xi$ and $\langle \cdot, \cdot \rangle_R^D$ be given in Section~\ref{sec:prelim}. Let $\{X,Y,Z\}$ be the basis of the Heisenberg Lie algebra $\mathfrak{h}$ with $[X,Y]=cZ$.
For $\xi\in \Xi$, let  $\nabla^0:\Xi \to \Xi \otimes \mathfrak{h}^\ast$  be a linear map given by
\begin{equation}\label{App:eq:min_conn}
\begin{split}
&(\nabla^0_X\xi)(x,y)=-\frac{\partial \xi}{\partial y}(x,y)+\frac{\pi ci}{2\mu}x^2f(x,y)\\
&(\nabla^0_Y\xi)(x,y)=-\frac{\partial \xi}{\partial x}(x,y)\\
&(\nabla^0_Z\xi)(x,y)=\frac{\pi i x}{\mu}\xi(x,y).
\end{split}
\end{equation}
Then $\nabla^0$ is a compatible linear connection on $\Xi$ with respect to the $\DD$-valued inner product $\langle \cdot, \cdot \rangle_R^D$.
\end{prop}

\begin{proof}
We have to show that $\nabla^0$ satisfies \eqref{eq:nabla-der} and \eqref{eq:nabla-comp}.
So fix $f,g\in \Xi$, $\Phi\in \DD$, then compute
\[\begin{split}
&(\nabla^0_X(f)\cdot \Phi)(x,y)+(f\cdot \delta_X(\Phi))(x,y)\\
&=\sum_q \nabla^0_X(f)(x+2q\mu,y+2q\nu)\overline{\Phi}(x+2q\mu,y+2q\nu,q)\\
&\quad +\sum_q f(x+2q\mu,y+2q\nu)\overline{\delta_X(\Phi)}(x+2q\mu,y+2q\nu,q)\\
&=\sum_q \Big(-\frac{\partial f}{\partial y}(x+2q\mu,y+2q\nu)+\frac{\pi ci}{2\mu}(x+2q\mu)^2 f(x+2q\mu,y+2q\nu)\Big)\overline{\Phi}(x+2q\mu,y+2q\nu,q)\\
\end{split}\]
\[\begin{split}
&\quad +\sum_q f(x+2q\mu,y+2q\nu)\Big(\overline{2\pi icq(x+2q\mu-q\mu)\Phi(x+2q\mu,y+2q\nu,q)}-\frac{\partial \overline{\Phi}}{\partial y}(x+2q\mu,y+2q\nu,q)\Big)\\
&=-\sum_q \frac{\partial f}{\partial y}(x+2q\mu,y+2q\nu)\overline{\Phi}(x+2q\mu,y+2q\nu,q) -\sum_q \Big(f(x+2q\mu,y+2q\nu)\\
&\quad \times \frac{\partial \overline{\Phi}}{\partial y}(x+2q\mu,y+2q\nu,q)\Big) + \sum_q \Big(f(x+2q\mu,y+2q\nu)\overline{\Phi}(x+2q\mu,y+2q\nu,q)\\
&\quad \times \big(\frac{\pi ci}{2\mu}(x+2q\mu)^2-2\pi icq(x+q\mu)\big)\Big)\\
&=-\sum_q \frac{\partial f}{\partial y}(x+2q\mu,y+2q\nu)\overline{\Phi}(x+2q\mu,y+2q\nu,q) -\sum_q \Big(f(x+2q\mu,y+2q\nu)\\
&\quad \times \frac{\partial \overline{\Phi}}{\partial y}(x+2q\mu,y+2q\nu,q)\Big) + \frac{\pi ci}{2\mu}x^2 \sum_q f(x+2q\mu,y+2q\nu)\overline{\Phi}(x+2q\mu,y+2q\nu,q).
\end{split}\]
On the other hand,
\[\begin{split}
&\nabla^0_X(f\cdot \Phi)(x,y)=-\frac{\partial}{\partial y}(f\cdot \Phi)(x,y)+\frac{\pi ci}{2\mu}x^2(f\cdot \Phi)(x,y)\\
&=-\frac{\partial}{\partial y}\Big(\sum_q f(x+2q\mu,y+2q\nu)\overline{\Phi}(x+2q\mu,y+2q\nu,q)\Big) \\
&\quad\quad + \frac{\pi ci}{2\mu}x^2\sum_q f(x+2q\mu,y+2q\nu)\overline{\Phi}(x+2q\mu,y+2q\nu,q)\\
&=-\sum_q \frac{\partial f}{\partial y}(x+2q\mu,y+2q\nu)\overline{\Phi}(x+2q\mu,y+2q\nu,q) -\sum_q \Big(f(x+2q\mu,y+2q\nu)\\
&\quad \times \frac{\partial \overline{\Phi}}{\partial y}(x+2q\mu,y+2q\nu,q)\Big) + \frac{\pi ci}{2\mu}x^2 \sum_q f(x+2q\mu,y+2q\nu)\overline{\Phi}(x+2q\mu,y+2q\nu,q)
\end{split}\]
Thus $\nabla^0_X(f\cdot \Phi)(x,y)=(\nabla^0_X(f)\cdot \Phi)(x,y)+(f\cdot \delta_X(\Phi))(x,y)$.

For $\nabla^0_Y$, we compute
\[\begin{split}
&(\nabla^0_Y(f)\cdot \Phi)(x,y)+(f\cdot \delta_Y(\Phi))(x,y)\\
&=\sum_q \nabla^0_Y(f)(x+2q\mu,y+2q\nu)\overline{\Phi}(x+2q\mu,y+2q\nu,q)\\
&\quad +\sum_q f(x+2q\mu,y+2q\nu)\overline{\delta_Y(\Phi)}(x+2q\mu,y+2q\nu,q)\\
\end{split}\]
\[\begin{split}
&=-\sum_q \frac{\partial f}{\partial x}(x+2q\mu, y+2q\nu)\overline{\Phi}(x+2q\mu, y+2q\nu,q)\\
&\quad -\sum_q f(x+2q\mu, y+2q\nu)\frac{\partial \overline{\Phi}}{\partial x}(x+2q\mu, y+2q\nu,q)\\
&=\nabla^0_Y(f\cdot \Phi)(x,y).
\end{split}\]

For $\nabla^0_Z$, we compute
\[\begin{split}
&(\nabla^0_Z(f)\cdot \Phi)(x,y)+(f\cdot \delta_Z(\Phi))(x,y)\\
&=\sum_q \nabla_Z(f)(x+2q\mu,y+2q\nu)\overline{\Phi}(x+2q\mu,y+2q\nu,q)\\
&\quad +\sum_q f(x+2q\mu,y+2q\nu)\overline{\delta_Z(\Phi)}(x+2q\mu,y+2q\nu,q)\\
&=\sum_q \frac{\pi i(x+2q\mu)}{\mu}f(x+2q\mu,y+2q\nu)\overline{\Phi}(x+2q\mu,y+2q\nu,q)\\
&\quad +\sum_q f(x+2q\mu,y+2q\nu)(\overline{2\pi i q}\,\overline{\Phi}(x+2q\mu,y+2q\nu,q))\\
&=\frac{\pi ix}{\mu}(f\cdot \Phi)(x,y)=\nabla^0_Z(f\cdot \Phi)(x,y).
\end{split}\]
Thus $\nabla^0$ is a connection on $\Xi$.

To show that $\nabla^0$ is compatible with respect to $\langle \cdot, \cdot\rangle_R^D$, we compute
\[\begin{split}
&\langle \nabla^0_X(f), g\rangle_R^D(x,y,p)+\langle f,\nabla^0_X(g)\rangle_R^D(x,y,p)\\
&=\sum_k \overline{e}(ckp(y-p\nu))\nabla^0_X(f)(x+k,y)\overline{g}(x-2p\mu+k,y-2p\nu)\\
&\quad + \sum_k \overline{e}(ckp(y-p\nu)) f(x+k,y)\overline{\nabla^0_X(g)}(x-2p\mu+k,y-2p\nu)\\
&=\sum_k \overline{e}(ckp(y-p\nu))\Big(-\frac{\partial f}{\partial y}(x+k,y)+\frac{\pi ci}{2\mu}(x+k)^2 f(x+k,y)\Big)\overline{g}(x-2p\mu+k,y-2p\nu)\\
&\quad + \sum_k \overline{e}(ckp(y-p\nu))f(x+k,y)\Big(-\frac{\partial \overline{g}}{\partial y}(x-2p\mu+k,y-2p\nu)\\
&\quad +\overline{\frac{\pi ci}{2\mu}}(x-2p\mu+k)^2\overline{g}(x-2p\mu+k,y-2p\nu)\Big)\\
&=-\sum_k \overline{e}(ckp(y-p\nu))\frac{\partial f}{\partial y}(x+k,y)\overline{g}(x-2p\mu+k,y-2p\nu)\\
&\quad - \sum_k \overline{e}(ckp(y-p\nu)) f(x+k,y)\frac{\partial \overline{g}}{\partial y}(x-2p\mu+k,y-2p\nu)\\
&\quad + \sum_k 2\pi icp (x-2p\mu+k)\overline{e}(ckp(y-p\nu)) f(x+k,y)\overline{g}(x-2p\mu+k,y-2p\nu).
\end{split}\]
On the other hand,
\[\begin{split}
&\delta_X(\langle f, g\rangle_R^D)(x,y,p)\\
&=2\pi icp(x-p\mu)\langle f, g\rangle_R^D(x,y,p)-\frac{\partial}{\partial y}\big(\langle f, g\rangle_R^D\big)(x,y,p)\\
\end{split}\]
\[\begin{split}
&=2\pi cip(x-p\mu)\langle f, g\rangle_R^D(x,y,p)-\frac{\partial}{\partial y}\big(\sum_k \overline{e}(ckp(y-p\nu))f(x+k,y)\overline{g}(x-2p\mu+k,y-2p\nu)\Big)\\
&=2\pi cip(x-p\mu) \sum_k \overline{e}(ckp(y-p\nu))f(x+k,y)\overline{g}(x-2p\mu+k,y-2p\nu)\\
&\quad-\sum_k (-2\pi ickp)\overline{e}(ckp(y-p\nu))f(x+k,y)\overline{g}(x-2p\mu+k,y-2p\nu)\\
&\quad -\sum_k \overline{e}(ckp(y-p\nu))\frac{\partial f}{\partial y}(x+k,y)\overline{g}(x-2p\mu+k,y-2p\nu)\\
&\quad - \sum_k \overline{e}(ckp(y-p\nu)) f(x+k,y)\frac{\partial \overline{g}}{\partial y}(x-2p\mu+k,y-2p\nu)\\
&=\sum_k 2\pi icp (x-2p\mu+k)\overline{e}(ckp(y-p\nu)) f(x+k,y)\overline{g}(x-2p\mu+k,y-2p\nu).\\
&\quad -\sum_k \overline{e}(ckp(y-p\nu))\frac{\partial f}{\partial y}(x+k,y)\overline{g}(x-2p\mu+k,y-2p\nu)\\
&\quad - \sum_k \overline{e}(ckp(y-p\nu)) f(x+k,y)\frac{\partial \overline{g}}{\partial y}(x-2p\mu+k,y-2p\nu).\\
\end{split}\]
Thus $\langle \nabla^0_X(f), g\rangle_R^D(x,y,p)+\langle f,\nabla^0_X(g)\rangle_R^D(x,y,p)=\delta_X(\langle f, g\rangle_R^D)(x,y,p)$.

For $\nabla^0_Y$, we compute
\[\begin{split}
&\langle \nabla^0_Y(f), g\rangle_R^D(x,y,p)+\langle f,\nabla^0_Y(g)\rangle_R^D(x,y,p)\\
&=\sum_k \overline{e}(ckp(y-p\nu))\nabla^0_Y(f)(x+k,y)\overline{g}(x-2p\mu+k,y-2p\nu)\\
&\quad + \sum_k \overline{e}(ckp(y-p\nu)) f(x+k,y)\overline{\nabla^0_Y(g)}(x-2p\mu+k,y-2p\nu)\\
&=-\sum_k \overline{e}(ckp(y-p\nu))\frac{\partial f}{\partial x}(x+k,y)\overline{g}(x-2p\mu+k,y-2p\nu)\\
&\quad -\sum_k \overline{e}(ckp(y-p\nu)) f(x+k,y)\frac{\partial \overline{g}}{\partial x}(x-2p\mu+k,y-2p\nu).\\
\end{split}\]
On the other hand,
\[\begin{split}
&\delta_Y(\langle f, g\rangle_R^D)(x,y,p)=-\frac{\partial}{\partial x}(\langle f, g\rangle_R^D)(x,y,p)\\
&=-\frac{\partial}{\partial x}\Big(\sum_k \overline{e}(ckp(y-p\nu)) f(x+k,y)\overline{g}(x-2p\mu+k,y-2p\nu)\Big)\\
&=-\sum_k \overline{e}(ckp(y-p\nu))\frac{\partial f}{\partial x}(x+k,y)\overline{g}(x-2p\mu+k,y-2p\nu)\\
&\quad -\sum_k \overline{e}(ckp(y-p\nu)) f(x+k,y)\frac{\partial \overline{g}}{\partial x}(x-2p\mu+k,y-2p\nu).\\
\end{split}\]
Thus $\langle \nabla^0_Y(f), g\rangle_R^D(x,y,p)+\langle f,\nabla^0_Y(g)\rangle_R^D(x,y,p)=\delta_Y(\langle f, g\rangle_R^D)(x,y,p)$.

For $\nabla^0_Z$, we compute
\[\begin{split}
&\langle \nabla^0_Z(f), g\rangle_R^D(x,y,p)+\langle f,\nabla^0_Z(g)\rangle_R^D(x,y,p)\\
&=\sum_k \overline{e}(ckp(y-p\nu))\nabla^0_Z(f)(x+k,y)\overline{g}(x-2p\mu+k,y-2p\nu)\\
&\quad\quad\quad + \sum_k \overline{e}(ckp(y-p\nu)) f(x+k,y)\overline{\nabla^0_Z(g)}(x-2p\mu+k,y-2p\nu)\\
\end{split}\]
\[\begin{split}
&=\sum_k \overline{e}(ckp(y-p\nu))\Big(\frac{\pi i(x+k)}{\mu}f(x+k,y)\Big)\overline{g}(x-2p\mu+k,y-2p\nu)\\
&\quad +\sum_k \overline{e}(ckp(y-p\nu)) f(x+k,y)\Big(\frac{\overline{\pi i(x-2p\mu+k)}}{\mu}\overline{g}(x-2p\mu+k,y-2p\nu)\\
&=2\pi ip \Big(\sum_k \overline{e}(ckp(y-p\nu)) f(x+k,y)\overline{g}(x-2p\mu+k,y-2p\nu)\Big)\\
&=2 \pi ip \langle f, g\rangle_R^D(x,y,p)=\delta_Z(\langle f, g\rangle_R^D)(x,y,p).
\end{split}\]
Therefore, $\nabla^0$ is a compatible connection on $\Xi$ with respect to $\langle \cdot, \cdot \rangle_R^D$.
\end{proof}

To verify the curvature given in \eqref{eq:min_curv}, recall from \cite{CR} that the values of $\Theta_\nabla$ of a compatible connection $\nabla$ on $\Xi$ are in ${(\EE)}^{s}$, the set of skew-symmetric elements of $\EE$. i.e. $\Theta^{\ast}_\nabla(X,Y)=-\Theta_\nabla(X,Y)$ for all $X,Y\in \mathfrak{h}$.
Also recall from Proposition~7 of \cite{Kang1} that for a multiplication-type element $\mathbb{G}$ of $\EE$ with corresponding function $G\in C^\infty(\mathbb{T}^2)$, i.e. $\mathbb{G}(x,y,p)=G(x,y)\delta_0(p)$, $\mathbb{G}$ is skew-symmetric if and only if 
\[(\mathbb{G}\cdot f)(x,y)=-G(x,y) f(x,y) \quad \text{for $f\in \Xi$.}\]
Thus for given $\nabla^0$ in Proposition~\eqref{App:eq:min_conn}, once we obtain
\[
(\Theta_{\nabla^0}(Y,Z)\cdot f)(x,y)=-\frac{\pi i}{\mu} f(x,y) \quad \text{for $f\in \Xi$,}
\]
which will be shown in the proof of the following proposition, then we get  $\Theta_{\nabla^0}(Y,Z)=\frac{\pi i}{\mu}\Id_E$ by Proposition~7 of \cite{Kang1}.

\begin{prop}\label{propA:curv}
The compatible connection $\nabla^0$ given in Proposition~\ref{App:prop:cc} has the following form of constant curvature:
\begin{equation}
\Theta_{\nabla^0}(X,Y)=0,\;\;\Theta_{\nabla^0}(X,Z)=0,\;\;\Theta_{\nabla^0}(Y,Z)=\frac{\pi i}{\mu}\Id_E,
\end{equation}
where $\Id_E(x,y,p)=\delta_0(p)$.
\end{prop}

\begin{proof}
We compute the curvature $\Theta_{\nabla^0}$ as follows. Fix $f\in \Xi$, then compute
\[\begin{split}
&(\Theta_{\nabla^0}(X,Y)\cdot f)(x,y)=\nabla^0_X(\nabla^0_Yf)(x,y)-\nabla^0_Y(\nabla^0_Xf)(x,y)-(\nabla^0_{[X,Y]}f)(x,y)\\
&=-\frac{\partial}{\partial y}(\nabla^0_Yf)(x,y)+\frac{\pi ci}{2\mu}x^2 (\nabla^0_Yf)(x,y)+\frac{\partial}{\partial x}(\nabla^0_Xf)(x,y)-c\frac{\pi ix}{\mu}f(x,y)\\
\end{split}\]
\[\begin{split}
&=-\frac{\partial}{\partial y}\Big(-\frac{\partial f}{\partial x}(x,y)\Big)+ \frac{\pi ci}{2\mu}x^2\Big(-\frac{\partial f}{\partial x}(x,y)\Big)+\frac{\partial}{\partial x}\Big(-\frac{\partial f}{\partial y}(x,y) +\frac{\pi ci}{2\mu}x^2f(x,y)\Big)\\
&\quad\quad-\frac{\pi cix}{\mu}f(x,y)\\
&=\frac{\partial^2 f}{\partial y\partial x}(x,y)-\frac{\pi ci}{2\mu}x^2\frac{\partial f}{\partial x}(x,y)-\frac{\partial^2 f}{\partial y\partial x}(x,y)+\frac{\pi ci}{\mu} xf(x,y)+\frac{\pi ci}{2\mu}x^2\frac{\partial f}{\partial x}(x,y)\\
&\quad \quad -\frac{\pi cix}{\mu}f(x,y)\\
&=0.
\end{split}\]
Thus $\Theta_{\nabla^0}(X,Y)=0$. Also
\[\begin{split}
&(\Theta_{\nabla^0}(X,Z)\cdot f)(x,y)=\nabla^0_X(\nabla^0_Zf)(x,y)-\nabla^0_Y(\nabla^0_Xf)(x,y)\\
&=-\frac{\partial}{\partial y}(\nabla^0_Zf)(x,y)+\frac{\pi ci}{2\mu}x^2(\nabla^0_Zf)(x,y)-\frac{\pi ix}{\mu}(\nabla^0_Xf)(x,y)\\
&=-\frac{\partial}{\partial y}\Big(\frac{\pi ix}{\mu}f(x,y)\Big)+\frac{\pi ci}{2\mu}x^2\Big(\frac{\pi ix}{\mu}f(x,y)\Big)-\frac{\pi ix}{\mu}\Big(-\frac{\partial f}{\partial y}(x,y)+\frac{\pi ci}{2\mu}x^2f(x,y)\Big)\\
&=-\frac{\pi ix}{\mu}\frac{\partial f}{\partial y}(x,y)+\frac{\pi ci}{2\mu}x^2\Big(\frac{\pi ix}{\mu}f(x,y)\Big)+\frac{\pi ix}{\mu}\frac{\partial f}{\partial y}(x,y)-\frac{\pi ix}{\mu}\Big(\frac{\pi ci}{2\mu}x^2f(x,y)\Big)\\
&=0.
\end{split}\]
Thus $\Theta_{\nabla^0}(X,Z)=0$. Finally
\[\begin{split}
&(\Theta_{\nabla^0}(Y,Z)\cdot f)(x,y)=\nabla^0_Y(\nabla^0_Zf)(x,y)-\nabla^0_Z(\nabla^0_Yf)(x,y)\\
&=-\frac{\partial}{\partial x}(\nabla^0_Zf)(x,y)-\frac{\pi ix}{\mu}(\nabla^0_Yf)(x,y)\\
&=-\frac{\partial}{\partial x}\Big(\frac{\pi ix}{\mu}f(x,y)\Big)-\frac{\pi ix}{\mu}\Big(-\frac{\partial f}{\partial x}(x,y)\Big)\\
&=-\frac{\pi i}{\mu}f(x,y)-\frac{\pi ix}{\mu}\frac{\partial f}{\partial x}(x,y)+\frac{\pi ix}{\mu}\frac{\partial f}{\partial x}(x,y)\\
&=-\frac{\pi i}{\mu}f(x,y).
\end{split}\]
Thus by Proposition~7 of \cite{Kang1} we have $\Theta_{\nabla^0}(Y,Z)=\frac{\pi i}{\mu}\Id_E$, which completes the proof.

\end{proof}

\end{appendices}

\end{document}